\documentclass[a4paper, 11pt]{article}
\usepackage{mystyle}

\begin{document}
	\onehalfspacing
	
	\title{A characteristic map for the holonomy groupoid of a foliation} 
	\author{Lachlan E. MacDonald\\School of Mathematics and Applied Statistics\\
		University of Wollongong\\
		Northfields Ave, Wollongong, NSW, 2522}

	\date{November 2019}
	
	\maketitle
	
	\begin{abstract}
		We prove a generalisation of Bott's vanishing theorem for the full transverse frame holonomy groupoid of any transversely orientable foliated manifold.  As a consequence we obtain a characteristic map encoding both primary and secondary characteristic classes.  Previous descriptions of this characteristic map are formulated for the Morita equivalent \'{e}tale groupoid obtained via a choice of complete transversal. By working with the full holonomy groupoid we obtain novel geometric representatives of characteristic classes.  In particular we give a geometric, non-\'{e}tale analogue of the codimension 1 Godbillon-Vey cyclic cocycle of Connes and Moscovici in terms of path integrals of the curvature form of a Bott connection.
	\end{abstract}

\section{Introduction}

Characteristic classes for foliation groupoids have been studied in an \'{e}tale context in \cite{cyctrans,hopf1,goro1,diffcyc,goro2,crainic1,moscorangi, backindgeom,mosc1}.  In this paper, we give analogous constructions in the context of the non-\'{e}tale, or ``full", holonomy groupoid of a foliation.  In working with such groupoids we stay close to Bott's Chern-Weil construction of the characteristic classes of a foliation \cite{bott2}, and obtain a novel, global geometric interpretation of the Godbillon-Vey cyclic cocycle of Connes and Moscovici \cite[Proposition 19]{backindgeom}.

Associated to any (regular) foliated manifold $(M,\FF)$, of codimension $q$, with leafwise tangent bundle $T\FF$ and normal bundle $N=TM/T\FF$, are characteristic classes associated to $N$ living in the de Rham cohomology $H^{*}_{dR}(M)$ of $M$.  Among these classes are the usual Pontryagin classes for the real vector bundle $N$, as well as certain \emph{secondary classes}, such as the Godbillon-Vey invariant, which is tied to the dynamical behaviour of the foliation $\FF$ \cite{hh}.

Classically, representatives of these classes are obtained either via Gelfand-Fuks cohomology \cite{bott3} or via Chern-Weil theory \cite{bott2}.  The Chern-Weil theory in particular makes critical use of connections on $N$ that, in any foliated coordinate system, coincide with the trivial connection on $N$ along leaves. Such connections are called \emph{Bott connections}, named after their originator R. Bott \cite{bott1}.  Bott connections give rise to \emph{Bott's vanishing theorem}, which states that elements of degree greater than $2q$ in the Pontryagin ring of $N$ must vanish.  It is precisely this vanishing phenomenon that guarantees the existence of the secondary classes.

To any such foliated manifold $(M,\FF)$ is associated its \emph{holonomy groupoid} $\GG$.  As a space, $\GG$ may be thought of as a quotient of the space of all smooth paths in leaves of $\FF$, with two such paths identifying in $\GG$ if and only if their parallel transport maps, defined with respect to any Bott connection on $N$, coincide.  The holonomy groupoid $\GG$ carries a natural, locally Hausdorff differentiable structure, and is thus a Lie groupoid \cite{wink}.  Moreover the normal bundle $N$ carries a natural action of $\GG$, and one may therefore be interested in extending the ``static" characteristic classes appearing in $H^{*}_{dR}(M)$ to ``dynamic" characteristic classes for the holonomy groupoid $\GG$.

In order to obtain such classes, it has become standard practice in the literature \cite{cyctrans,hopf1,goro1,diffcyc,goro2,crainic1,moscorangi, backindgeom,mosc1} to ``\'{e}talify" the holonomy groupoid as follows.  Letting $q$ be the codimension of $(M,\FF)$, one takes any $q$-dimensional submanifold $\TT\subset M$ which intersects each leaf of $\FF$ at least once, and which is everywhere transverse to $\FF$ in the sense that $T_{x}\TT\oplus T_{x}\FF = T_{x}M$ for all $x\in \TT$.  Such a submanifold is called a \emph{complete transversal} for $(M,\FF)$.  Having chosen such a complete transversal $\TT$ we consider the subgroupoid
\[
\GG^{\TT}_{\TT}:=\{u\in\GG:r(u),s(u)\in\TT\}
\]
of $\GG$.  The subgroupoid $\GG^{\TT}_{\TT}$ inherits from $\GG$ a differential topology for which it is a (generally non-Hausdorff) \emph{\'{e}tale Lie groupoid} \cite[Lemma 2]{crainic2} - that is, a Lie groupoid whose range (and therefore source) are local diffeomorphisms.  For any choice of complete transversal $\TT$, the $C^{*}$-algebras of the groupoids $\GG$ and $\GG_{\TT}^{\TT}$ are Morita equivalent \cite{hs2,ben7}, so are the same as far as $K$-theory is concerned.  Moreover, the groupoids $\GG$ and $\GG^{\TT}_{\TT}$ are themselves Morita equivalent \cite[Lemma 2]{crainic2}, and consequently they are (co)homologically identical \cite{crainic2, crainic1} also.

This paper provides an analogous characteristic map to those defined in \cite{diffcyc,crainic1} in the context of the \emph{full} holonomy groupoid $\GG$ of a foliated manifold.  Our characteristic map will be constructed in a Chern-Weil fashion from Bott connections for $N$, staying as close as possible to the classical geometric approach.  Section \ref{sc1} consists of necessary background on differential graded algebras, Weil algebras, classical characteristic classes for foliations, and Chern-Weil theory for Lie groupoids as developed in \cite{chgpd}.  This background material suffices to access the Pontryagin classes of $N$, for which a choice of connection form $\alpha$ on the positively oriented frame bundle $\Fr^{+}(N)$ of $N$ determines a characteristic map $\psi^{\GG}_{\alpha}$ sending the invariant polynomials on $\mathfrak{gl}(q,\RB)$ to the de Rham cohomology $\Omega^{*}(\GG^{(*)}_{1})$ of the action groupoid $\GG_{1}:=\GG\ltimes\Fr^{+}(N)$.

In order to access the \emph{secondary classes} of the foliation, such as the Godbillon-Vey invariant, in Section \ref{sc2} we introduce and prove the following generalisation of Bott's vanishing theorem.

\begin{thm}\label{bottintro}
	Let $\alpha^{\flat}$ be the connection form associated to a Bott connection on $N$.  Then the image of $\psi^{\GG}_{\alpha^{\flat}}$ in $\Omega^{*}(\GG_{1}^{(*)})$ vanishes in total degree greater than $2q$.
\end{thm}

The vanishing of certain cocycles implied by Theorem \ref{bottintro} enables us to refine $\psi_{\alpha^{\flat}}$ to a characteristic map of the truncated Weil algebra so as to obtain secondary classes, in a manner entirely analogous to the classical setting.  More precisely we have the following theorem.

\begin{thm}\label{charintro}
	Let $\underline{WO}_{q}$ denote the truncated Weil algebra.  If $\alpha^{\flat}$ is the connection form associated to any Bott connection on $N$, then we obtain a characteristic map
	\[
	\psi^{\GG}_{\alpha^{\flat}}:\underline{WO}_{q}\rightarrow\Omega^{*}(\GG^{(*)}_{1}/\SO(q,\RB)).
	\]
\end{thm}

Theorems \ref{bottintro} and \ref{charintro} should be thought of as the non-\'{e}tale analogues of \cite[Theorem 2 (iv)]{crainic1} and \cite[Lemma 17]{diffcyc} respectively.  We conclude Section \ref{sc2} by discussing how the lack of an invariant Euclidean structure on $N$ obstructs the na\"{i}ve construction of a characteristic map for the de Rham complex $\Omega^{*}(\GG^{(*)})$ associated to $\GG$.

Finally in Section \ref{sc3}, we restrict ourselves to the codimension 1 case, and use the explicit formulae provided by Theorem \ref{charintro} to derive a Godbillon-Vey cyclic cocycle for the convolution algebra $C_{c}^{\infty}(\GG_{1};\Omega^{\frac{1}{2}})$ of smooth leafwise half-densities $\Omega^{\frac{1}{2}}$ associated to $\GG_{1}$.  Our cocycle should be thought of as a non-\'{e}tale analogue of the Connes-Moscovici formula \cite[Proposition 19]{backindgeom}.  In contrast with the Connes-Moscovici formula obtained in the \'{e}tale setting, the cocycle we obtain in this paper has a novel geometric interpretation in terms of path integrals of the curvature form associated to a Bott connection.  These facts can be summarised as follows.

\begin{thm}
	Let $R^{\flat}$ be the curvature form on $\Fr^{+}(N)$ associated to a Bott connection on $N$, and let $R^{\GG}$ denote the differential 1-form on $\GG_{1}$ defined by
	\begin{equation}\label{intcurvintro}
	R^{\GG}_{u}:=\int_{\gamma}R^{\flat},\hspace{7mm}u\in\GG_{1}
	\end{equation}
	where $\gamma$ is any leafwise path in $\Fr^{+}(N)$ representing $u$. Assume that $\Fr^{+}(N)\cong M\times\RB^{*}_{+}$ has been trivialised by a choice of trivialisation for $N$.  Then the formula
	\[
	\varphi_{gv}(a^{0},a^{1}):=\int_{(x,t)\in\Fr^{+}(N)}\int_{u\in(\GG_{1})_{(x,t)}}a^{0}(u^{-1})\,a^{1}(u)\,\frac{dt}{t}\wedge R^{\GG}_{u},\hspace{7mm}a^{0},a^{1}\in C_{c}^{\infty}(\GG_{1};\Omega^{\frac{1}{2}})
	\]
	defines a cyclic cocycle $\varphi_{gv}$ for the convolution algebra $C_{c}^{\infty}(\GG_{1};\Omega^{\frac{1}{2}})$.
\end{thm}

We conclude the paper by demonstrating that the cyclic cocycle $\varphi_{gv}$ coincides with that obtained as Chern character of a semifinite spectral triple constructed using groupoid equivariant $KK$-theory in \cite[Section 4.3]{macr1}.  In doing so we give a (non-\'{e}tale) geometric interpretation for the off-diagonal term appearing in the triangular structures considered by Connes \cite[Lemma 5.2]{cyctrans} and Connes-Moscovici \cite[Part I]{CM}, in terms of the integrated curvature of Equation \eqref{intcurvintro}.

Let us stress that the approach taken in this paper has the advantage of being \emph{intrinsically geometric}, giving representatives of cohomological data that are expressed in terms of \emph{global} geometric data for $(M,\FF)$.  For instance, the Godbillon-Vey cyclic cocycle obtained in this paper has a completely novel interpretation in terms of line integrals of the Bott curvature over paths in $\FF$ representing elements of $\GG$ (see Proposition \ref{integralprop}). This is to be contrasted with the approaches taken in the \'{e}tale context, in which the geometry of $M$ has necessarily been lost by ``chopping up" $\GG$ into $\GG^{\TT}_{\TT}$.  In the \'{e}tale context, explicit formulae have so far tended to be obtained by tracking the displacement of \emph{local} geometric data (trivial connections in local transversals) \cite[Section 5.1, p. 47]{crainic1, backindgeom} under the action of $\GG^{\TT}_{\TT}$, which will in general not be easily relatable to the global geometry of $M$.

\subsection{Acknowledgements}

I wish to thank the Australian Federal Government for a Research Training Program scholarship.  I also thank Moulay Benameur for supporting a visit to Montpellier in late 2018, and Magnus Goffeng for supporting a visit to Gothenburg in early 2019, where parts of this research were conducted.  I also thank Magnus Goffeng and James Stasheff for helpful comments on the paper.  Finally, I would like to extend deep thanks to Adam Rennie, whose consistent (but never overbearing) guidance and support have greatly benefited my growth as a mathematician.

\section{Background}\label{sc1}

\subsection{Differential graded algebras and the Weil algebra}

For the entirety of this subsection, denote by $G$ a Lie group with Lie algebra $\mathfrak{g}$.  One of the key tools in Chern-Weil theory is the notion of a $G$-differential graded algebra.

\begin{defn}\label{gdga}
	A \textbf{$G$-differential graded algebra} $(A,d,i)$ is a differential graded algebra $(A,d)$ equipped with a $G$-action that preserves the grading of $A$ and commutes with the differential, and a linear map $i$ sending $\mathfrak{g}$ to the derivations of degree -1 on $(A,d)$ such that for all $X,Y\in\mathfrak{g}$ and for all $g\in G$ one has
		\[
		i_{X}\circ i_{Y} = -i_{Y}\circ i_{X},
		\]
		\[
		g\circ i_{X}\circ g^{-1} = i_{\Ad_{g}(X)},
		\]
		and
		\[
		i_{X}\circ d+ d\circ i_{X} = L_{X},
		\]
	where $L_{X}$ denotes the infinitesimal $G$-action defined by
	\[
	L_{X}(a):=\frac{d}{dt}\bigg|_{t=0}(\exp(tX)\cdot a),\hspace{7mm}a\in A.
	\]
	If $(A,d,i^{A})$ and $(B,b,i^{B})$ are two $G$-differential graded algebras, a homomorphism $\phi:A\rightarrow B$ of differential graded algebras is said to be a \textbf{homomorphism of $G$-differential graded algebras} if it commutes with the action of $G$ and with the maps $i^{A}$ and $i^{B}$
\end{defn}

Since $G$-differential graded algebras are in particular differential graded algebras, they admit a natural cochain complex and associated cohomology.  For geometric applications however, one is usually more interested in the associated \emph{basic} cohomology.

\begin{defn}\label{basic}
	Let $(A,d,i)$ be a $G$-differential graded algebra, and suppose that $K$ is a Lie subgroup of $G$, with Lie algebra $\mathfrak{k}$.  We say that an element $a\in A$ is \textbf{$G$-invariant} if $g\cdot a = a$ for all $g\in G$.  We say that $a\in A$ is \textbf{$K$-basic} if it is $G$-invariant and if $i_{X}a = 0$ for all $X\in\mathfrak{k}$, and denote the space of $K$-basic elements by $A_{K-basic}$.  The $G$-basic elements will be referred to simply as \textbf{basic}.
\end{defn}

If $(A,d,i)$ is a $G$-differential graded algebra, then it is an easy consequence of the commutativity of $d$ with the action of $G$, as well as the fact that $i_{X}\circ d+d\circ i_{X} = L_{X}$ for all $X\in\mathfrak{g}$, that $d$ preserves the space $A_{K-basic}$.  Therefore we obtain a cochain complex
\[
0\rightarrow A^{0}_{K-basic}\xrightarrow{d} A^{1}_{K-basic}\xrightarrow{d} A^{2}_{K-basic}\xrightarrow{d}\cdots
\]
whose cohomology we will be mostly interested in for geometric applications.

\begin{defn}
	Let $(A,d,i)$ be a $G$-differential graded algebra, and let $K$ be a Lie subgroup of $G$.  For each $n\in\NB$, the $n^{th}$ \textbf{$K$-basic cohomology group} of $(A,d,i)$ is the group
	\[
	H^{n}_{K-basic}(A):=\ker(d:A^{n}_{K-basic}\rightarrow A^{n+1}_{K-basic})/\im(d:A^{n-1}_{K-basic}\rightarrow A^{n}_{K-basic}).
	\]
	The \textbf{$K$-basic cohomology} of $(A,d,i)$ is the collection $H^{*}_{K-basic}(A)$ of all $K$-basic cohomology groups.
\end{defn}

Because a homomorphism $\phi:A_{1}\rightarrow A_{2}$ of $G$-differential graded algebras commutes with the respective actions of $G$ and $\mathfrak{g}$, it naturally induces a homomorphism $\phi:H^{*}_{K-basic}(A_{1})\rightarrow H^{*}_{K-basic}(A_{2})$ of their $K$-basic cohomologies for any Lie subgroup $K$ of $G$.  It will be useful to have an analogue of cochain homotopy to decide when the maps of cohomology induced by two such $G$-differential graded algebra homomorphisms coincide.

\begin{defn}
	Let $\phi_{0},\phi_{1}:A_{1}\rightarrow A_{2}$ be two homomorphisms of $G$-differential graded algebras $(A_{1},d_{1},i^{1})$ and $(A_{2},d_{2},i^{2})$.  We say that $\phi_{0}$ and $\phi_{1}$ are \textbf{$G$-cochain homotopic} if there exists a cochain homotopy $C:A_{1}^{*}\rightarrow A^{*-1}_{2}$ for which $C\circ g = g\circ C$ and $C\circ i^{1}_{X} = i^{2}_{X}\circ C$ for all $g\in G$ and $X\in\mathfrak{g}$.
\end{defn}

The equivariance properties of $G$-cochain homotopies can be used to verify the following homotopy invariance result.

\begin{prop}
	If $\phi_{0}$ and $\phi_{1}$ are $G$-cochain homotopic homomorphisms $A_{1}\rightarrow A_{2}$ of $G$-differential graded algebras then the maps $H^{*}_{K-basic}(A_{1})\rightarrow H^{*}_{K-basic}(A_{2})$ induced by $\phi_{0}$ and $\phi_{1}$ coincide for all Lie subgroups $K$ of $G$.\qed
\end{prop}

Let us now consider some important examples.

\begin{ex}\label{princ1}\normalfont
	Just as the immediate geometric example of a differential graded algebra is the algebra of differential forms on a manifold, the first geometric example of a $G$-differential graded algebra is given by the algebra of differential forms $\Omega^{*}(P)$ on the total space of a principal $G$-bundle $\pi:P\rightarrow M$ over a manifold $M$.
	
	The algebra $(\Omega^{*}(P),d)$ is regarded as a differential graded algebra in the usual way, while the action of $G$ on $\Omega^{*}(P)$ is obtained by pulling back differential forms under the canonical right action $R:P\times G\rightarrow P$.  That is, the action of $g\in G$ on $\omega\in\Omega^{*}(P)$ is given by
	\[
	g\cdot\omega:=R_{g^{-1}}^{*}\omega.
	\]
	To obtain the linear map $i$ from $\mathfrak{g}$ into the derivations of degree -1 on $\Omega^{*}(P)$ we recall that any $X\in\mathfrak{g}$ is associated with the fundamental vector field $V^{X}$ on $P$ defined by
	\[
	V^{X}_{p}:=\frac{d}{dt}\bigg|_{t=0}(p\cdot\exp(tX)),\hspace{7mm}p\in P.
	\]
	For any $X\in\mathfrak{g}$ we then define $i_{X}$ on $\Omega^{*}(P)$ to be the interior product operator with $V^{X}$.  The usual properties of the interior product together with the fact that
	\[
	(dR_{g})_{p}(V^{X}_{p}) = V^{\Ad_{g^{-1}}(X)}_{p\cdot g}
	\]
	for all $p\in P$ and $g\in G$ then show that $(\Omega^{*}(P),d,i)$ is indeed a $G$-differential graded algebra.  Since $P/G\cong M$, the basic complex $(\Omega^{*}(P)_{basic},d)$ identifies naturally with $(\Omega^{*}(M),d)$.  More generally, if $K$ is a Lie subgroup of $G$ with Lie algebra $\mathfrak{k}$, then $K$-basic elements of $\Omega^{*}(P)$ of degree $m$ are precisely those forms $\tilde{\omega}\in\Omega^{m}(P)$ for which
	\[
	\tilde{\omega}(X_{1}+\mathfrak{k},\dots,X_{m}+\mathfrak{k}) = \tilde{\omega}(X_{1},\dots,X_{m})
	\]
	is well-defined for all $X_{1},\dots,X_{m}\in TP$.  Any $\omega\in\Omega^{*}(P/K)$ pulls back therefore to a $K$-basic form on $P$, while any $K$-basic form $\tilde{\omega}$ on $P$ determines a form $\omega$ on $P/K$ by the formula
	\[
	\omega(X_{1},\dots,X_{m}):=\tilde{\omega}(X_{1}+\mathfrak{k},\dots,X_{m}+\mathfrak{k}).
	\]
	Thus the space of $K$-basic elements in $\Omega^{*}(P)$ coincides with the differential graded algebra $\Omega^{*}(P/K)$.
\end{ex}

Chern-Weil theory is obtained for principal $G$-bundles over manifolds by considering connection and curvature forms.  Our next example will be key in formalising the properties of such forms.

\begin{ex}\label{weil1}\normalfont
	One of the most important examples of a $G$-differential graded algebra is the \emph{Weil algebra} $W(\mathfrak{g})$ associated to the Lie algebra $\mathfrak{g}$ of $G$ \cite[Chapter 3]{sed}.  The Weil algebra should be thought of as being the home of ``universal" connection and curvature forms, and is constructed as follows.
	
	For each $k\in\NB$, denote by $S^{k}(\mathfrak{g}^{*})$ the space of functions
	\[
	\underbrace{\mathfrak{g}\times\cdots\times\mathfrak{g}}_{k\text{ times}}\rightarrow\RB
	\]
	which are invariant under the action of the symmetric group on the $k$ factors.  Denote by $S(\mathfrak{g}^{*})$ the sum over $k\in\NB$ of the $S^{k}(\mathfrak{g}^{*})$.  Consider also the exterior algebra $\Lambda(\mathfrak{g}^{*})$ defined in the usual way.  The \emph{Weil algebra} associated to $\mathfrak{g}$ is
	\[
	W(\mathfrak{g}):=S(\mathfrak{g}^{*})\otimes\Lambda(\mathfrak{g}^{*}).
	\]
	We endow $W(\mathfrak{g})$ with a grading by declaring any element $a\otimes b\in S^{k}(\mathfrak{g}^{*})\otimes\Lambda^{l}(\mathfrak{g}^{*})$ to have degree $2k+l$, under which $W(\mathfrak{g})$ is a graded-commutative algebra.  
	
	To define a differential on $W(\mathfrak{g})$ we choose a basis $(X_{i})_{i=1}^{\dim(\mathfrak{g})}$ for $\mathfrak{g}$ (for a basis-free definition, see \cite[Section 3.1]{sed}), with associated structure constants $f^{i}_{jk}$ defined by the equation
	\[
	[X_{i},X_{j}] = \sum_{i}f^{k}_{ij}X_{k}.
	\]
	The corresponding dual basis $(\xi^{i})_{i=1}^{\dim(\mathfrak{g})}$ of $\mathfrak{g}^{*}$ determines generators $\omega^{i}:=1\otimes\xi^{i}\in S^{0}(\mathfrak{g}^{*})\otimes\Lambda^{1}(\mathfrak{g}^{*})$ of degree 1 and $\Omega^{i}:=\xi^{i}\otimes1\in S^{1}(\mathfrak{g}^{*})\otimes\Lambda^{0}(\mathfrak{g}^{*})$ of degree 2, with respect to which the differential $d$ is defined by
	\[
	d\Omega^{i}:=\sum_{j,k}f^{i}_{jk}\Omega^{j}\omega^{k} \hspace{5mm} d\omega^{i} := \Omega^{i}-\frac{1}{2}\sum_{j,k}f^{i}_{jk}\omega^{j}\omega^{k}
	\]
	Extending $d$ to all of $W(\mathfrak{g})$ turns $(W(\mathfrak{g}),d)$ into the graded-commutative differential graded algebra that is freely generated by the $\omega^{i}$ and $\Omega^{i}$.  Note that by definition of $d\omega^{i}$, we can equally regard $W(\mathfrak{g})$ as being freely generated by the $\omega^{i}$ and $d\omega^{i}$.
	
	The coadjoint action of $G$ on $\mathfrak{g}^{*}$ given by
	\[
	(g\cdot\xi)(X):=(\Ad_{g^{-1}}^{*}\xi)(X) = \xi(\Ad_{g^{-1}}X)
	\]
	for $g\in G$, $\xi\in\mathfrak{g}^{*}$ and $X\in\mathfrak{g}$ extends to an action of $G$ on the generators $\alpha^{i}$, $\Omega^{i}$ and hence to an action of $G$ on all of $W(\mathfrak{g})$.  For $X\in\mathfrak{g}$, we define a derivation $i_{X}$ of degree -1 by
	\[
	i_{X}(\Omega^{i}) := 0 \hspace{5mm} i_{X}(\omega^{i}) := \omega^{i}(X),
	\]
	and the corresponding map $i$ from $\mathfrak{g}$ to the derivations of degree -1 of $W(\mathfrak{g})$ satisfies the required properties to make $(W(\mathfrak{g}),d,i)$ a $G$-differential graded algebra.
	
	By definition of $i$, the basic elements of $W(\mathfrak{g})$ identify with the space $I(G) = S(\mathfrak{g}^{*})^{G}$ of symmetric polynomials that are invariant under the coadjoint action of $G$.  If more generally $K$ is a Lie subgroup of $G$, then we use $W(\mathfrak{g},K)$ to denote the subalgebra of $K$-basic elements.
\end{ex}

The notions of connection and curvature can be formulated at the abstract algebraic level of $G$-differential graded algebras.

\begin{defn}\label{algcon}
	Let $(A,d,i)$ be a $G$-differential graded algebra.  A \textbf{connection} on $A$ is a degree 1 element $\alpha\in A^{1}\otimes\mathfrak{g}$ such that:
	\begin{enumerate}
		\item $g\cdot\alpha = \alpha$ for all $g\in G$, where the action of $g$ on $A\otimes\mathfrak{g}$ is given by $g\cdot(a\otimes X) = (g\cdot a)\otimes\Ad_{g}(X)$ and,
		\item $i_{X}\alpha = 1\otimes X$ for all $X\in\mathfrak{g}$.
	\end{enumerate}
	If $\alpha$ is a connection on $A$, then its \textbf{curvature} is the degree 2 element $R\in A^{2}\otimes\mathfrak{g}$ defined by the formula
	\[
	R:=d\alpha+\frac{1}{2}[\alpha,\alpha],
	\]
	where $[\cdot,\cdot]$ denotes graded Lie bracket of elements in $A\otimes\mathfrak{g}$.
\end{defn}

\begin{ex}\label{princon}\normalfont
	Let $\pi:P\rightarrow M$ be a principal $G$-bundle, with associated $G$-differential graded algebra $(\Omega^{*}(P),d,i)$ as defined in Example \ref{princ1}.  Then a connection $\alpha$ on $\Omega^{*}(P)$ is precisely a connection 1-form $\alpha\in\Omega^{1}(P;\mathfrak{g})$.  By definition, such a connection 1-form is invariant under the action of $G$
	\[
	g\cdot\alpha = \Ad_{g}(R_{g^{-1}}^{*}\alpha) = \alpha,\hspace{7mm}g\in G,
	\]
	and is vertical in the sense that $\alpha(V^{X}) = X$ for all $X\in\mathfrak{g}$, where $V^{X}$ denotes the fundamental vector field on $P$ associated to $X$.  The curvature of $\alpha$ is of course just the usual curvature 2-form $R\in\Omega^{2}(P;\mathfrak{g})$ defined by $\alpha$.
\end{ex}

\begin{ex}\label{weilcon}\normalfont
	The Weil algebra $W(\mathfrak{g})$ constructed in Example \ref{weil1} admits a canonical connection.  Given a basis $(X_{i})_{i=1}^{\dim(\mathfrak{g})}$ for $\mathfrak{g}$, with associated dual basis $(\xi^{i})_{i=1}^{\dim(\mathfrak{g})}$ and generators $\omega^{i} = 1\otimes\xi^{i}$ of degree 1 and $\Omega^{i} = \xi^{i}\otimes 1$ of degree 2 respectively, we define $\omega\in W(\mathfrak{g})^{1}\otimes\mathfrak{g}$ by the formula
	\[
	\omega:=\sum_{i}\omega^{i}\otimes X_{i}.
	\]
	Because the $X_{i}$ transform covariantly and the $\omega^{i}$ contravariantly this $\omega$ does not depend on the basis chosen, and for the same reason is invariant under the action of $G$.  By construction we have $i_{X}\omega= 1\otimes X$ for all $X\in\mathfrak{g}$.
\end{ex}

The Weil algebra enjoys the following universal property as a classifying algebra for connections on $G$-differential graded algebras.

\begin{thm}\cite[Theorem 3.3.1]{sed}\label{weiluni}
	Let $(A,d)$ be a $G$-differential graded algebra, and suppose that $\alpha\in A^{1}\otimes\mathfrak{g}$ is a connection on $A$.  Then there exists a unique homomorphism $\phi_{\alpha}:W(\mathfrak{g})\rightarrow A$ of $G$-differential graded algebras such that $\phi(\omega) = \alpha$.  Moreover if $\alpha_{0}$, $\alpha_{1}$ are two different connections on $A$, the corresponding maps $\phi_{\alpha_{0}}$ and $\phi_{\alpha_{1}}$ are $G$-cochain homotopic.\qed
\end{thm}

In this paper we will mostly be interested in the case $G = \GL^{+}(q,\RB)$ of invertible $q\times q$ matrices with positive determinant, and where $K = \SO(q,\RB)$.  Regard an element $\xi$ of $\mathfrak{gl}(q,\RB)^{*}$ as a matrix in the usual way, and let $\Omega:=\xi\otimes 1\in S^{1}(\mathfrak{gl}(q,\RB)^{*})\otimes\Lambda^{0}(\mathfrak{gl}(q,\RB)^{*})$ and $\omega:=1\otimes\xi\in S^{0}(\mathfrak{gl}(q,\RB)^{*})\otimes\Lambda^{1}(\mathfrak{gl}(q,\RB)^{*})$ denote the corresponding elements of $W(\mathfrak{gl}(q,\RB))$ of degrees 2 and 1 respectively.  Using matrix multiplication, the differential on $W(\mathfrak{gl}(q,\RB))$ acts by the simple formulae
\[
d\omega=\Omega-\omega^{2},\hspace{7mm}d\Omega = \Omega\omega-\omega\Omega.
\]
It is well-known \cite[p. 187]{gue} that for $1\leq i\leq q$, the elements $c_{i}\in W(\mathfrak{gl}(q,\RB))$ defined by
\begin{equation}\label{ci}
c_{i}:=\Tr(\Omega^{i})
\end{equation}
are all cocycles and are all $\GL^{+}(q,\RB)$-basic.  With respect to the decomposition $\mathfrak{gl}(q,\RB) = \mathfrak{so}(q,\RB)\oplus\mathfrak{s}(q,\RB)$ of all $q\times q$ matrices into antisymmetric matrices and symmetric matrices respectively, our elements $\omega$ and $\Omega$ defined above decompose as
\[
\omega = \omega_{o}+\omega_{s},\hspace{7mm}\Omega = \Omega_{o}+\Omega_{s}.
\]
Here the subscript $o$ denotes the antisymmetric part, while the subscript $s$ denotes the symmetric part.  It can then be shown \cite[Proposition 5]{gue} that for $1\leq i\leq q$, the elements
\[
h_{i}:=i\Tr\bigg(\int_{0}^{1}\omega_{s}(t\Omega_{s}+\Omega_{o}+(t^{2}-1)\omega^{2}_{s})^{i-1}dt\bigg)
\]
of $W(\mathfrak{gl}(q,\RB))$ satisfy $dh_{i} = c_{i}$, and are $\SO(q,\RB)$-basic for $i$ odd.  We can assemble the $c_{i}$ and $h_{i}$ into a new, simpler subalgebra of $W(\mathfrak{gl}(q,\RB),\SO(q,\RB))$ which is useful for obtaining explicit formulae.

\begin{defn}\label{WOq}
	Let $q\geq 1$ and let $o$ be the largest odd integer that is less than or equal to $q$.  We denote by $WO_{q}$ the differential graded subalgebra of $W(\mathfrak{gl}(q,\RB),\SO(q,\RB))$ generated by odd degree elements $h_{1},h_{3},\dots,h_{o}$, with $\deg(h_{i}) = i$, and even degree elements $c_{1},c_{2},\dots,c_{q}$ with $\deg(c_{i}) = 2i$.
\end{defn}

We will abuse terminology in referring to the algebra $WO_{q}$ as the Weil algebra.  We pre-empt its use for foliations with the following refinement of Theorem \ref{weiluni}.

\begin{cor}\label{glweiluni}
	Let $(A,d)$ be a $\GL^{+}(q,\RB)$-differential graded algebra, and suppose that $\alpha\in A\otimes\mathfrak{gl}(q,\RB)$ is a connection on $A$, with curvature $R$.  Decompose $\alpha=\alpha_{o}+\alpha_{s}$ and $R = R_{o}+R_{s}$ into their antisymmetric and symmetric components respectively.  Then the formulae
	\[
	c_{i}\mapsto \Tr(R^{i}),\hspace{7mm}\text{for $1\leq i\leq q$},
	\]
	\[
	h_{i}\mapsto i\Tr\bigg(\int_{0}^{1}\alpha_{s}(tR_{s}+R_{o}+(t^{2}-1)\alpha_{s}^{2})^{i-1}dt\bigg),\hspace{7mm}\text{ for $1\leq i\leq q$, $i$ odd},
	\]
	define a homomorphism $\psi_{\alpha}:WO_{q}\rightarrow A_{\SO(q,\RB)-basic}$ of differential graded algebras.  Moreover if $\beta$ is any other choice of connection on $A$, the maps induced by $\psi_{\alpha}$ and $\psi_{\beta}$ on cohomology coincide.\qed
\end{cor}

\subsection{The classical Chern-Weil homomorphism for foliations}\label{ssc2}

Recall that a foliated manifold $(M,\FF)$ of codimension $q$ is \emph{transversely orientable} if its normal bundle $\pi_{N}:N:=TM/T\FF\rightarrow M$ is an orientable vector bundle.  Given such a foliated manifold, we can mimic the classical Chern-Weil construction of Bott \cite{bott2} using Corollary \ref{glweiluni} as follows.   Let $\pi_{\Fr^{+}(N)}:\Fr^{+}(N)\rightarrow M$ denote the positively oriented transverse frame bundle of $N$, a principal $\GL^{+}(q,\RB)$-bundle whose fibre $\Fr^{+}(N)_{x}$ over $x\in M$ consists of all positively oriented linear isomorphisms $\phi:\RB^{q}\rightarrow N_{x}$.  Letting $\GG$ denote the holonomy groupoid \cite{wink} of $(M,\FF)$, recall \cite[Section 2.2]{macr1} that there is a natural action $\GG\times_{s,\pi_{N}}N\rightarrow N$ of $\GG$ on $N$ by linear isomorphisms, which we denote
\[
\GG\times_{s,\pi_{N}}N\ni(u,n)\mapsto u_{*}n\in N.
\]
We obtain an induced action $\GG\times_{s,\pi_{\Fr^{+}(N)}}\Fr^{+}(N)\rightarrow\Fr^{+}(N)$ of $\GG$ on $\Fr^{+}(N)$ defined by
\[
u\cdot\phi:=u_{*}\circ\phi:\RB^{q}\rightarrow N_{r(u)},\hspace{7mm}(u,\phi)\in\GG\times_{s,\pi_{\Fr^{+}(N)}}\Fr^{+}(N).
\]
By associativity of composition, this action of $\GG$ commutes with the canonical right action of $\GL^{+}(q,\RB)$ on the principal $\GL^{+}(q,\RB)$-bundle $\Fr^{+}(N)$.  Moreover, the orbits of $\GG$ in $\Fr^{+}(N)$ define a foliation $\FF_{\Fr^{+}(N)}$ of $\Fr^{+}(N)$ for which the differential of the projection $\pi_{\Fr^{+}(N)}$ maps $T\FF_{\Fr^{+}(N)}$ fibrewise-isomorphically onto $T\FF$.  The following definition does not appear explicitly in the literature, but as we will show it is nonetheless essentially classical.

\begin{defn}\label{bottconform}
	A connection form $\alpha^{\flat}\in\Omega^{*}(\Fr^{+}(N);\mathfrak{gl}(q,\RB))$ is called a \textbf{Bott connection form} if $T\FF_{\Fr^{+}(N)}\subset\ker(\alpha^{\flat})$.
\end{defn}

Let us justify this terminology.  First, let $p:TM\rightarrow N$ denote the projection onto the normal bundle, and recall that a connection $\nabla^{\flat}:\Gamma^{\infty}(M;N)\rightarrow\Gamma^{\infty}(M;T^{*}M\otimes N)$ for $N$ is called a \emph{Bott connection} if it satisfies
\begin{equation}\label{bottconclass}
\nabla^{\flat}_{X}p(Y) = p[X,Y], \hspace{7mm}Y\in\Gamma^{\infty}(M;TM),
\end{equation}
whenever $X$ is a leafwise vector field.  The next result establishes the relationship between Bott connections in the sense of Equation \eqref{bottconclass} and Bott connection forms in the sense of Definition \ref{bottconform}.  This relationship is essentially classical, so although the next result does not appear explicitly in the literature and its proof is independent, we make no claim to originality.

\begin{prop}\label{bottconnections}
	Bott connections $\nabla^{\flat}$ on $N$ are in bijective correspondence with Bott connection forms $\alpha^{\flat}$ on $\Fr^{+}(N)$.  Moreover, any Bott connection form $\alpha^{\flat}$ on $\Fr^{+}(N)$ canonically determines a Bott connection $\nabla^{\Fr^{+}(N)}$ for the foliated manifold $(\Fr^{+}(N),\FF_{\Fr^{+}(N)})$.
\end{prop}

\begin{proof}
	For the first part, let $\nabla$ be a connection on $N$ and let $\alpha$ be the connection form on $\Fr^{+}(N)$ determined by $\nabla$.  Notice that in foliated coordinates $(x^{1},\dots,x^{p};z^{1},\dots,z^{q})$ over $U\subset M$, defining a local section $\chi_{U}:U\rightarrow\Fr^{+}(N)|_{U}$ of $\pi_{\Fr^{+}(N)}$, $\nabla$ can be written
	\[
	\nabla = d+\alpha_{U}
	\]
	where $\alpha_{U}:=\chi_{U}^{*}\alpha\in\Omega^{1}(U;\mathfrak{gl}(q,\RB))$.  Let $\sigma = \sigma^{i}\frac{\partial}{\partial z^{i}}$ be a normal vector field over $U$, and let $X$ be a leafwise vector field over $U$.  Then
	\[
	\nabla_{X}\sigma = d\sigma^{i}(X)\frac{\partial}{\partial z^{i}}+\alpha_{U}(X)\sigma = p[X,\sigma]+\alpha_{U}(X)\sigma
	\]
	is equal to $p[X,\sigma]$ if and only if $\alpha_{U}(X)$ vanishes.  Thus a connection $\nabla$ on $N$ is a Bott connection if and only if its local connection form $\alpha_{U}$ vanishes on leafwise vectors in any foliated coordinate neighbourhood $U$.  Since $\pi_{\Fr^{+}(N)}$ maps $T\FF_{\Fr^{+}(N)}$ fibrewise isomorphically to $T\FF$ we see that every $\alpha_{U}$ vanishes on $T\FF$ if and only if $\alpha$ vanishes on $T\FF_{\Fr^{+}(N)}$.  Consequently $\nabla =\nabla^{\flat}$ is a Bott connection on $N$ if and only if its associated connection form $\alpha = \alpha^{\flat}$ on $\Fr^{+}(N)$ is a Bott connection form.
	
	For the second part, suppose that $\alpha^{\flat}$ is a Bott connection form on $\Fr^{+}(N)$.  By hypothesis we have $T\FF_{\Fr^{+}(N)}\subset\ker(\alpha^{\flat})$, and since $\ker(\alpha^{\flat})$ projects fibrewise-isomorphically onto $TM$, the quotient bundle $H:=\ker(\alpha^{\flat})/T\FF_{\Fr^{+}(N)}$ projects fibrewise-isomorphically onto $N = TM/T\FF$.  Consequently, $H$ admits a tautological trivialisation $H\cong\Fr^{+}(N)\times\RB^{q}$ defined by
	\[
	H_{\phi}\ni h\mapsto\phi^{-1}(d\pi_{\Fr^{+}(N)}(h))\in\RB^{q},\hspace{7mm}\phi\in\Fr^{+}(N).
	\]
	The vertical bundle $V:=\ker(d\pi_{\Fr^{+}(N)})$ over $\Fr^{+}(N)$ is canonically trivialised by the fundamental vector fields, so we obtain the canonical trivialisation
	\[
	N_{\Fr^{+}(N)}:=T\Fr^{+}(N)/T\FF_{\Fr^{+}(N)} = V\oplus H\cong \Fr^{+}(N)\times(\RB^{q^{2}}\oplus\RB^{q})
	\]
	of the normal bundle for the foliated manifold $(\Fr^{+}(N),\FF_{\Fr^{+}(N)})$.  With respect to this global trivialisation, the vanishing of $\alpha^{\flat}$ on $T\FF_{\Fr^{+}(N)}$ implies that
	\[
	\nabla^{\Fr^{+}(N)}:=d+\id_{\RB^{q^{2}}}\oplus\,\alpha^{\flat}
	\]
	defines a Bott connection on $N_{\Fr^{+}(N)}$.
\end{proof}

Bott connections are important because of the following vanishing result, known as Bott's vanishing theorem \cite[p. 34]{bott1}.  Its proof follows from an easy local coordinate calculation.

\begin{thm}[Bott's vanishing theorem]
	Let $\alpha^{\flat}\in\Omega^{1}(\Fr^{+}(N);\mathfrak{gl}(q,\RB))$ be a Bott connection form, and let $R^{\flat}:=d\alpha^{\flat}+\alpha^{\flat}\wedge\alpha^{\flat}$ be its curvature.  Then any polynomial of degree greater than $q$ in the components of $R^{\flat}$ vanishes.\qed
\end{thm}

Bott's vanishing theorem motivates the following refinement of the Weil algebra $WO_{q}$.

\begin{defn}
	Let $\mathcal{J}$ be the differential ideal in $WO_{q}$ consisting of elements of $\RB[c_{1},\dots,c_{q}]\subset WO_{q}$ that are of degree greater than $2q$.  The \textbf{truncated Weil algebra} is the quotient $\underline{WO}_{q}:=WO_{q}/\mathcal{J}$.
\end{defn}

The following result is now an immediate consequence of Bott's vanishing theorem, together with Corollary \ref{glweiluni}.

\begin{thm}\label{classchar}
	A choice of Bott connection form $\alpha^{\flat}$ determines a homomorphism $\psi_{\alpha^{\flat}}:WO_{q}\rightarrow\Omega^{*}(\Fr^{+}(N)/\SO(q,\RB))$ of differential graded algebras, that factors through the truncated Weil algebra $\underline{WO}_{q}$.  That is, letting $p:WO_{q}\rightarrow\underline{WO}_{q}$ denote the projection, there is a homomorphism $\phi_{\alpha^{\flat}}:\underline{WO}_{q}\rightarrow\Omega^{*}(\Fr^{+}(N)/\SO(q,\RB))$ of differential graded algebras such that the diagram
	\begin{center}
		\begin{tikzcd}
			WO_{q} \ar[r,"\psi_{\alpha^{\flat}}"] \ar[d,"p"] & \Omega^{*}(\Fr^{+}(N)/\SO(q,\RB)) \\ \underline{WO}_{q} \ar[ur,"\phi_{\alpha^{\flat}}" below]
		\end{tikzcd}
	\end{center}
	commutes.  The maps on cohomology induced by $\psi_{\alpha^{\flat}}$ and $\phi_{\alpha^{\flat}}$ do not depend on the Bott connection chosen.\qed
\end{thm}

Given a choice of Bott connection form $\alpha^{\flat}$ on $\Fr^{+}(N)$, those classes determined by the range of $\phi_{\alpha^{\flat}}$ that are not contained in the Pontryagin ring $[\phi_{\alpha}(\RB[c_{1},\dots,c_{q}])]\subset H^{ev}_{dR}(M)$ for $N$ are called \emph{secondary characteristic classes}.  In particular, the \emph{Godbillon-Vey class} is the class $[\phi_{\alpha^{\flat}}(h_{1}c_{1}^{q})]\in H^{2q+1}_{dR}(\Fr^{+}(N)/\SO(q,\RB))$.

\begin{rmk}\label{pldn}\normalfont
	The fibre $\GL^{+}(q,\RB)/\SO(q,\RB)$ of $\Fr^{+}(N)/\SO(q,\RB)$ is contractible, so the total space of $\Fr^{+}(N)/\SO(q,\RB)$ has the same cohomology as $M$.  More specifically, a choice of Euclidean metric on $N$ determines a smooth section $\sigma:M\rightarrow\Fr^{+}(N)/\SO(q,\RB)$, which, together with a choice of Bott connection $\alpha^{\flat}$ on $\Fr^{+}(N)$, determines a characteristic map $\sigma^{*}\circ\phi_{\alpha^{\flat}}:\underline{WO}_{q}\rightarrow\Omega^{*}(M)$ for $M$.  The arguments of \cite[Remarque (c)]{gue} show that this characteristic map agrees on the level of differential forms with that defined by Bott \cite{bott2}.
\end{rmk}

\subsection{Chern-Weil homomorphism for Lie groupoids}

The groupoid Chern-Weil material we present in this subsection is sourced primarily from the paper \cite{chgpd}, whose historical antecedents are to be found in the papers \cite{dupont2,bss}.

Just as the classical Chern-Weil theory can be simplified and systematised by using principal $G$-bundles, Chern-Weil theory at the level of Lie groupoids is most easily studied using principal bundles over groupoids.  For the entirety of this section we let $\GG$ be a (not necessarily Hausdorff) Lie groupoid, with unit space $\GG^{(0)}$ and range and source maps $r,s$ respectively, and let $G$ be a Lie group.

\begin{defn}
	A \textbf{principal $G$-bundle over $\GG$} consists of a principal $G$-bundle $\pi:\PP^{(0)}\rightarrow\GG^{(0)}$ over $\GG^{(0)}$ together with an action $\sigma:\PP:=\GG\times_{s,\pi}\PP^{(0)}\rightarrow \PP^{(0)}$ that commutes with the right action of $G$ on $\PP^{(0)}$.  We will often refer to $\PP$ as a \textbf{principal $G$-groupoid} over $\GG$.
\end{defn}

Let us for the rest of this subsection fix a principal $G$-bundle $\pi:\PP^{(0)}\rightarrow\GG^{(0)}$ over $\GG$.  We will primarily consider the associated action groupoid $\PP:=\GG\times_{s,\pi}\PP^{(0)}$, and the resulting spaces $\PP^{(k)}$ of composable $k$-tuples of elements of $\PP$.  In order to make notation less cumbersome, for $(u_{1},\dots,u_{k})\in\GG^{(k)}$ and $p\in \PP^{(0)}_{s(u_{k})}$ we will denote the composable $k$-tuple
\[
\big((u_{1},(u_{2}\cdots u_{k})\cdot p),\,(u_{2},(u_{3}\cdots u_{k})\cdot p),\dots,(u_{k},p)\big)\in\PP^{(k)}
\]
by simply
\[
(u_{1},\dots,u_{k})\cdot p.
\]
The next result tells us that the $\PP^{(k)}$ fibre over the $\GG^{(k)}$ as principal $G$-bundles, and is a straightforward consequence of the fact that $\PP^{(0)}\rightarrow\GG^{(0)}$ is a principal $G$-bundle, together with the fact that the action of $G$ commutes with that of $\GG$.

\begin{lemma}\label{princgpd}
	Let $\pi:\PP^{(0)}\rightarrow\GG^{(0)}$ be a principal $G$-bundle over $\GG$.  Then for each $k\in\NB$ and $(u_{1},\dots,u_{k})\cdot p\in\PP^{(k)}$, the formula
	\[
	\pi^{(k)}((u_{1},\dots,u_{k})\cdot p):=(u_{1},\dots,u_{k})
	\]
	defines a principal $G$-bundle $\pi^{(k)}:\PP^{(k)}\rightarrow\GG^{(k)}$.\qed
\end{lemma}

\begin{rmk}\normalfont
	Note that the definition and properties of the exterior derivative $d$ on a manifold $Y$ depend only on the local structure of the manifold.  Consequently, the differential forms $(\Omega^{*}(Y),d)$ on $Y$ are a differential graded algebra even if $Y$ is only locally Hausdorff.  We will use this fact freely and without further comment in what follows.
\end{rmk}

We now have the following immediate consequence of Lemma \ref{princgpd} and Example \ref{princ1}.

\begin{cor}
	For all $k\in\NB$, the differential forms $\Omega^{*}(\PP^{(k)})$ on $\PP^{(k)}$ form a $G$-differential graded algebra.  If $K\subset G$ is a Lie subgroup, then the $K$-basic elements of $\Omega^{*}(\PP^{(k)})$ identify with the differential forms $\Omega^{*}(\PP^{(k)}/K)$ on the quotient of $\PP^{(k)}$ by the right action of $K$.\qed
\end{cor}

Let us now recall the definition of the de Rham cohomology of $\PP$ together with its relative versions.  Observe that there exist \emph{face maps} $\epsilon^{k}_{i}:\PP^{(k)}\rightarrow\PP^{(k-1)}$ defined for all $k>1$ and $0\leq i\leq k$ by the formulae
\[
\epsilon^{k}_{0}\big((u_{1},u_{2},\dots,u_{k})\cdot p\big):=(u_{2},\dots,u_{k})\cdot p,
\]
\[
\epsilon^{k}_{k}\big((u_{1},\dots,u_{k})\cdot p\big):=(u_{1},\dots,u_{k-1})\cdot (u_{k}\cdot p),
\]
and
\[
\epsilon^{k}_{i}\big((u_{1},\dots,u_{k})\cdot p\big):=(u_{1},\dots,u_{i}u_{i+1},\dots,u_{k})\cdot p
\]
for $1\leq i\leq k-1$.  For $k=1$, we obtain the range and source maps $\epsilon^{1}_{0}:=s:\PP\rightarrow \PP^{(0)}$ and $\epsilon^{1}_{1}:=r:\PP\rightarrow \PP^{(0)}$. 

\begin{rmk}\normalfont
	The collection $\{\PP^{(k)}\}_{k\geq 0}$, taken together with the face maps $\epsilon^{k}_{i}:\PP^{(k)}\rightarrow\PP^{(k-1)}$, is known in the literature as the \emph{nerve} $N\PP$ of $\PP$.  The nerve $N\PP$ is an example of a (semi) \emph{simplicial manifold} (the terminology used depends on the author).  It is in the general setting of simplicial manifolds that most of the technology in this section was developed by Bott-Shulman-Stasheff \cite{bss} and Dupont \cite{dupont,dupont2}.
\end{rmk}

The next result is an immediate consequence of the fact that the action of $\GG$ on $\PP^{(0)}$ commutes with the right action of $G$ on $\PP^{(0)}$.

\begin{lemma}\label{actionsonk}
	The face maps $\epsilon^{k}_{i}:\PP^{(k)}\rightarrow \PP^{(k-1)}$ commute with the actions of $G$ on $\PP^{(k)}$ and $\PP^{(k-1)}$.\qed
\end{lemma}

Since each $\PP^{(k)}$ is a manifold, the exterior derivative $d:\Omega^{*}(\PP^{(k)})\rightarrow\Omega^{*+1}(\PP^{(k)})$ is defined for all $k\geq0$ and satisfies the usual property $d^{2} = 0$.  The exterior derivative will form the vertical differential of the double complex from which we will construct the de Rham cohomology of $\PP$.  To obtain the horizontal differential, notice that the face maps $\epsilon^{k}_{i}:\PP^{(k)}\rightarrow\PP^{(k-1)}$ allow us to define a natural map $\partial:\Omega^{*}(\PP^{(k-1)})\rightarrow\Omega^{*}(\PP^{(k)})$ via an alternating sum of pullbacks
\[
\partial\omega:=\sum_{i=0}^{k}(-1)^{i}(\epsilon^{k}_{i})^{*}\omega.
\]
A routine calculation shows that $\partial^{2} = 0$.  By Lemma \ref{actionsonk} the coboundary mapping $\partial:\Omega^{*}(\PP^{(k-1)})\rightarrow\Omega^{*}(\PP^{(k)})$ preserves $K$-basic elements for any Lie subgroup $K$ of $G$, and we therefore obtain the double complex

\begin{center}
	\begin{tikzpicture}
	\matrix (m) [matrix of math nodes,row sep=3em,column sep=4em,minimum width=2em]
	{
		\vdots & \vdots & \vdots & \\
		\Omega^{1}(\PP^{(0)}/K) & \Omega^{1}(\PP^{(1)}/K) & \Omega^{1}(\PP^{(2)}/K) & {}\cdots\\
		\Omega^{0}(\PP^{(0)}/K) & \Omega^{0}(\PP^{(1)}/K) & \Omega^{0}(\PP^{(2)}/K) & \cdots\\
	};
	\path[-stealth]
	(m-2-1) edge node [right] {$d$} (m-1-1)
	(m-3-1) edge node [right] {$d$} (m-2-1)
	(m-2-2) edge node [right] {$d$} (m-1-2)
	(m-2-3) edge node [right] {$d$} (m-1-3)
	(m-3-2) edge node [right] {$d$} (m-2-2)
	(m-3-3) edge node [right] {$d$} (m-2-3)
	(m-2-1) edge node [above] {$\partial$} (m-2-2)
	(m-2-2) edge node [above] {$\partial$} (m-2-3)
	(m-2-3) edge node [above] {$\partial$} (m-2-4)
	(m-3-1) edge node [above] {$\partial$} (m-3-2)
	(m-3-2) edge node [above] {$\partial$} (m-3-3)
	(m-3-3) edge node [above] {$\partial$} (m-3-4);
	\end{tikzpicture}
\end{center}

\begin{defn}\label{gpdderham}
	Let $K$ be a Lie subgroup of $G$.  The double complex $(\Omega^{*}(\PP^{(*)}/K),d,\partial)$ is called the \textbf{$K$-basic de Rham complex} of the principal $G$-groupoid $\PP$.  The associated total complex is given by
	\[
	\Tot^{*}\Omega(\PP/K) = \bigoplus_{n+m = *}\Omega^{n}(\PP^{(m)}/K), \hspace{7mm} \delta|_{\Omega^{n}(\PP^{(m)}/K)}:=(-1)^{m}d+\partial.
	\]
	The cohomology of $(\Tot^{*}\Omega(\PP/K),\delta)$ is denoted by $H^{*}_{dR}(\PP/K)$ and is called the \textbf{$K$-basic de Rham cohomology} of the principal $G$-groupoid $\PP$.
\end{defn}

\begin{rmk}\normalfont
	The de Rham complex of Definition \ref{gpdderham} is frequently referred to in the literature as the \emph{Bott-Shulman-Stasheff complex} associated to the simplicial manifold $N(\PP/K)$, named after its originators \cite{bss}.  
\end{rmk}

In the same way that the exterior product of differential forms induces a multiplication in de Rham cohomology groups of any manifold, the exterior product of differential forms also induces a multiplication in the de Rham cohomology of $\GG$.

\begin{defn}Let $K$ be a Lie subgroup of $G$.  Given $\omega_{1}\in\Omega^{k}(\PP^{(m)}/K)$ and $\omega_{2}\in\Omega^{l}(\PP^{(n)}/K)$ we define the \textbf{cup product} $\omega_{1}\vee\omega_{2}\in\Omega^{k+l}(\PP^{(m+n)}/K)$ of $\omega_{1}$ and $\omega_{2}$ by
	\[
	(\omega_{1}\vee\omega_{2})_{(u_{1},\dots,u_{m+n})}:=(-1)^{kn}(p_{1}^{*}\omega_{1})_{(u_{1},\dots,u_{m+n})}\wedge (p_{2}^{*}\omega_{2})_{(u_{1},\dots,u_{m+n})},
	\]
	where $p_{1}:\PP^{(m+n)}/K\rightarrow\PP^{(m)}/K$ is given by
	\[
	p_{1}(u_{1},\dots,u_{m+n}) = \begin{cases}
	(u_{1},\dots,u_{m})&\text{ if $m\geq1$}\\
	r(u_{1})&\text{ if $m=0,  n\geq1$}\\
	\id&\text{ if $m=n=0$}
	\end{cases}
	\]
	and $p_{2}:\PP^{(m+n)}/K\rightarrow\PP^{(n)}/K$ is given by
	\[
	p_{2}(u_{1},\dots,u_{m+n}) = \begin{cases}
	(u_{m+1},\dots,u_{m+n})&\text{ if $n\geq1$}\\
	s(u_{m})&\text{ if $n=0,  m\geq1$}\\
	\id&\text{ if $m=n=0$}
	\end{cases}
	\]
\end{defn}

\begin{prop}
	The cup product descends to give a well-defined multiplication on the cohomology $H^{*}_{dR}(\PP/K)$, gifting $H^{*}_{dR}(\PP/K)$ the structure of a graded ring.
\end{prop}

\begin{proof}
	The cup product respects the bi-grading of $\Omega^{*}(\PP^{(*)}/K)$ by definition.  Well-definedness on $H^{*}_{dR}(\PP/K)$ follows from the Liebniz rule for the exterior derivative, and from noting that we may rewrite $p_{1}:\PP^{(m+n)}/K\rightarrow\PP^{(m)}/K$ as
	\[
	p_{1} = \epsilon^{m+1}_{m+1}\circ\cdots\circ\epsilon^{m+i}_{m+i}\circ\cdots\circ\epsilon^{m+n}_{m+n}
	\]
	and $p_{2}:\PP^{(m+n)}/K\rightarrow \PP^{(n)}/K$ as
	\[
	p_{2} = \epsilon^{n+1}_{0}\circ\cdots\circ\epsilon^{n+i}_{0}\circ\cdots\circ\epsilon^{n+m}_{0},
	\]
	where the $\epsilon^{k}_{i}:\PP^{(k)}\rightarrow\PP^{(k-1)}$ are the face maps.
\end{proof}

For our characteristic map we will also need a particular differential graded algebra which encodes the face maps of the $\PP^{(k)}$ into those of the standard simplices.  For $k\in\NB$, let $\Delta^{k}$ denote the standard $k$-simplex
\begin{equation}\label{simplex}
\Delta^{k}:=\bigg\{(t_{0},t_{1},\dots,t_{k})\in[0,1]^{k+1}:\sum_{i=1}^{k}t_{i} = 1\bigg\}.
\end{equation}
We have face maps $\tilde{\epsilon}^{k}_{i}:\Delta^{k-1}\rightarrow\Delta^{k}$ defined for all $k>1$ and $1\leq i\leq k$ by the formulae
\[
\tilde{\epsilon}^{k}_{i}(t_{0},\dots,t_{k-1}):=(t_{0},\dots,t_{i-1},0,t_{i},\dots,t_{k-1}).
\]
and for $i=0$ by simply
\[
\tilde{\epsilon}^{k}_{0}(t_{0},\dots,t_{k-1}):=(0,t_{0},\dots,t_{k-1}).
\]

\begin{defn}
	For $l\in\NB$, a \textbf{simplicial $l$-form} on $\PP$ is a sequence $\omega = \{\omega^{(k)}\}_{k\in\NB}$ of differential $l$-forms $\omega^{(k)}\in\Omega^{l}(\Delta^{k}\times\PP^{(k)})$ such that
	\[
	(\tilde{\epsilon}^{k}_{i}\times\id)^{*}\omega^{(k)} = (\id\times\epsilon^{k}_{i})^{*}\omega^{(k-1)}\in\Omega^{l}(\Delta^{(k-1)}\times\PP^{(k)})
	\]
	for all $i = 0,\dots,k$ and for all $k\in\NB$.  We denote the space of all simplicial $l$-forms on $\PP$ by $\Omega^{l}_{\Delta}(\PP)$.
\end{defn}

\begin{rmk}\normalfont
	One identifies $(t,\epsilon^{k}_{i}(v))\in\Delta^{k-1}\times\PP^{(k-1)}$ with $(\tilde{\epsilon}^{k}_{i}(t),v)\in\Delta^{k}\times\PP^{(k)}$ for all $k>0$, and defines the \emph{fat realisation} of $\|N\PP\|$ of $N\PP$ to be the space
	\[
	\|N\PP\|:=\bigsqcup_{k\geq0}\bigg(\Delta^{k}\times\PP^{(k)}\bigg)/\sim.
	\]
	The fat realisation is a geometric realisation of the \emph{classifying space} $B\PP$ of the groupoid $\PP$ \cite{segal}, and is not generally a manifold even though each of its ``layers" $\Delta^{k}\times\PP^{(k)}$ is.  Simplicial differential forms were defined by Dupont \cite{dupont2} so as to descend to ``forms on $B\PP$".  We will see shortly that together with the usual de Rham differential, simplicial differential forms define a differential graded algebra whose cohomology can be taken as the definition of the cohomology of the classifying space.
\end{rmk}

Importantly, simplicial differential forms on $\PP$ determine a differential graded algebra which will be instrumental in the construction of our characteristic map.  The next result follows from routine verification.

\begin{prop}\label{dgaofsimp}
	The wedge product and exterior derivative of differential forms on the manifolds $\Delta^{k}\times\PP^{(k)}$ together with the action of $G$ on the principal $G$-bundle $\PP^{(k)}$, make the space $\Omega^{*}_{\Delta}(\PP)$ of all simplicial differential forms on $\PP$ into a $G$-differential graded algebra.  If $K$ is any Lie subgroup of $G$, then the subcomplex $(\Omega^{*}_{\Delta}(\PP)_{K-basic},d)$ of $(\Omega^{*}_{\Delta}(\PP),d)$ coincides with the complex $(\Omega^{*}_{\Delta}(\PP/K),d)$ of simplicial differential forms on the groupoid $\PP/K$.\qed
\end{prop}

In order to relate simplicial differential forms to the de Rham cohomology of $\PP$, we make use of the integration over the fibres map.  The next result is proved in \cite[Theorem 2.3, Theorem 2.14]{dupont2}.

\begin{prop}\label{intoverfibre}
	Let $K$ be a Lie subgroup of $G$. Then the map $I:\Omega^{*}_{\Delta}(\PP/K)\rightarrow\Tot^{*}\Omega(\PP/K)$ defined by
	\[
	I(\omega):=\sum_{l\in\NB}\int_{\Delta^{l}}\omega^{(l)}
	\]
	is a map of cochain complexes.  Moreover the map determined by $I$ on cohomology is a homomorphism of rings, where the ring structure on $H^{*}(\Omega^{*}_{\Delta}(\PP/K))$ is induced by the wedge product and where the ring structure on $H^{*}_{dR}(\PP/K)$ is induced by the cup product.\qed
\end{prop}

\begin{rmk}\normalfont
	When $\GG$ is Hausdorff, the map $I:\Omega^{*}_{\Delta}(\PP/K)\rightarrow\Tot^{*}\Omega(\PP/K)$ descends to an isomorphism on cohomology \cite[Theorem 2.3]{dupont2}.  Thus for Hausdorff $\GG$ the double complex $\Omega^{*}(\GG^{(*)})$ computes the cohomology of the classifying space $B\GG$.
\end{rmk}

Before we give the characteristic map, we need to show that a connection form on $\PP^{(0)}$ induces a connection on the differential graded algebra $\Omega^{*}_{\Delta}(\PP)$.  The universal property of the Weil algebra $W(\mathfrak{g})$ (Theorem \ref{weiluni}) will then guarantee a homomorphism from $W(\mathfrak{g})$ to $\Omega^{*}_{\Delta}(\PP)$ which, composed with the cochain map $I$, will give us our characteristic map.

\begin{cons}\label{pki}\normalfont For each $0\leq i\leq k$, define $p^{k}_{i}:\PP^{(k)}\rightarrow\PP^{(0)}$ by
	\[
	p^{k}_{0}((u_{1},\dots,u_{k})\cdot p):=(u_{1}\cdots u_{k})\cdot p,
	\]
	\[
	p^{k}_{k}((u_{1},\dots,u_{k})\cdot p):=p,
	\]
	and
	\[
	p^{k}_{i}((u_{1},\dots,u_{k})\cdot p):=(u_{i+1}\cdots u_{k})\cdot p.
	\]
	for all $1\leq i \leq k-1$.  Since the range and source maps are $G$-equivariant, so too are the maps $p^{k}_{i}$.  Given a connection form $\alpha\in\Omega^{1}(\PP^{(0)};\mathfrak{g})$, for each $k\in\NB$ we define a differential form $\alpha^{(k)}\in\Omega^{1}(\Delta^{k}\times\PP^{(k)})$ by the formula
	\begin{equation}\label{bob}
	\alpha^{(k)}_{(t_{0},\dots,t_{k};(u_{1},\dots,u_{k})\cdot p)}:=\sum_{i=0}^{k}t_{i}((p^{k}_{i})^{*}\alpha)_{(u_{1},\dots,u_{k})\cdot p}.
	\end{equation}
\end{cons}

The next lemma now follows from routine calculations (see also \cite[Proposition 5.3]{chgpd}).

\begin{lemma}\label{con1}
	The sequence $\tilde{\alpha}:=\{\alpha^{(k)}\}_{k\in\NB}$ of 1-forms $\alpha^{(k)}\in\Omega^{1}(\Delta^{k}\times\PP^{(k)})\otimes\mathfrak{g}$ determines a connection on the differential graded algebra $\Omega^{*}_{\Delta}(\PP)$.\qed
\end{lemma}

Finally we can give the characteristic map as in \cite{chgpd}.

\begin{thm}\label{charmap}
	A choice of connection form $\alpha\in\Omega^{1}(\PP^{(0)};\mathfrak{g})$ determines, for any Lie subgroup $K$ of $G$, a homomorphism
	\[
	\varPsi^{\GG}_{\alpha}:W(\mathfrak{g},K)\rightarrow\Omega^{*}_{\Delta}(\PP/K)
	\]
	of differential graded algebras, hence a cochain map
	\[
	\psi^{\GG}_{\alpha} = I\circ\varPsi^{\GG}_{\alpha}:W(\mathfrak{g},K)\rightarrow \Tot^{*}\Omega(\PP/K)
	\]
	of total complexes.  The induced map on cohomology is a homomorphism of graded rings and does not depend on the connection chosen.
\end{thm}

\begin{proof}
	The existence of $\varPsi^{\GG}_{\alpha}$ follows from Lemma \ref{con1} together with Theorem \ref{weiluni}.  That $\psi^{\GG}_{\alpha}$ is a cochain map is true by Proposition \ref{intoverfibre}, while the cohomological independence of the choice of connection follows from Theorem \ref{weiluni}.  That $\psi^{\GG}_{\alpha}$ descends to a homomorphism $H^{*}(W(\mathfrak{g},K))\rightarrow H^{*}_{dR}(\PP/K)$ of graded rings follows from Proposition \ref{intoverfibre} together with the fact that $\varPsi^{\GG}_{\alpha}:W(\mathfrak{g},K)\rightarrow\Omega^{*}_{\Delta}(\PP/K)$ is a homomorphism of differential graded algebras.
\end{proof}

\section{Characteristic map for foliated manifolds}\label{sc2}

We will now use the background material of Section \ref{sc1} to prove new results on the secondary characteristic classes for the transverse frame holonomy groupoid of a foliation.  Let us consider a transversely orientable foliated manifold $(M,\FF)$ of codimension $q$ with holonomy groupoid $\GG$.  As we have already discussed in Subsection \ref{ssc2}, the principal $\GL^{+}(q,\RB)$-bundle $\pi_{\Fr^{+}(N)}:\Fr^{+}(N)\rightarrow M$ of positively oriented frames for $N$ carries a left action of the holonomy groupoid $\GG$ that commutes with the canonical right action of $\GL^{+}(q,\RB)$. 

\begin{defn}We will denote by $\GG_{1} = \GG\ltimes\Fr^{+}(N)$ the principal $\GL^{+}(q,\RB)$-groupoid over $\GG$ corresponding to the foliated principal $\GL^{+}(q,\RB)$-bundle $\pi_{\Fr^{+}(N)}:\Fr^{+}(N)\rightarrow M$.
\end{defn}

Given a connection form $\alpha\in\Omega^{1}(\Fr^{+}(N);\mathfrak{gl}(q,\RB))$, the characteristic map $\psi^{\GG}_{\alpha}$ of Theorem \ref{charmap} composes with the inclusion $WO_{q}\hookrightarrow W(\mathfrak{gl}(q,\RB),\SO(q,\RB))$ to give a cochain map $\psi^{\GG}_{\alpha}:WO_{q}\rightarrow\Omega^{*}(\GG_{1}^{(*)}/\SO(q,\RB))$, whose induced map on cohomology does not depend on the connection $\alpha$.  The image of $\psi^{\GG}_{\alpha}$ in the first-column subcomplex $\Omega^{*}(\Fr^{+}(N)/\SO(q,\RB))$ of $\Omega^{*}(\GG_{1}^{(*)}/\SO(q,\RB))$ coincides with that of the characteristic map $\psi_{\alpha}$ of Theorem \ref{classchar}.  More generally the image of $\psi^{\GG}_{\alpha}$ in $H^{*}_{dR}(\GG_{1}/\SO(q,\RB))$ consists of the Pontryagin classes of the groupoid $\GG_{1}/\SO(q,\RB)$, accessed in the same manner as in \cite{chgpd}.  In order to construct \emph{secondary} characteristic classes for the groupoid $\GG_{1}/\SO(q,\RB)$, we must prove an analogue of Bott's vanishing theorem - that is, we must prove that the Pontryagin classes of $\GG_{1}/\SO(q,\RB)$ vanish in total degree greater than $2q$.  To this end, we present the following generalisation of Bott's vanishing theorem, which is the non-\'{e}tale analogue of \cite[Theorem 2 (iv)]{crainic1}.  Regard elements of the subalgebra $I^{*}_{q}(\RB):=\RB[c_{1},\dots,c_{q}]\subset WO_{q}$ as invariant polynomials as in Equation \eqref{ci}.

\begin{thm}[Bott's vanishing theorem for $\GG_{1}$]\label{bottvan2}
	Let $\alpha^{\flat}\in\Omega^{1}(\Fr^{+}(N);\mathfrak{gl}(q,\RB))$ be a Bott connection form.  If $P\in I^{*}_{q}(\RB)$ is an invariant polynomial of degree $\deg(P)>q$ (so that its degree in $I^{*}_{q}(\RB)$ is greater than $2q$), then $\psi_{\alpha^{\flat}}(P) = 0\in\Omega^{*}(\GG^{(*)}_{1}/\SO(q,\RB))$.
\end{thm}

\begin{proof}
	For each $k\in\NB$, let $(R^{\flat})^{(k)} = d(\alpha^{\flat})^{(k)}+(\alpha^{\flat})^{(k)}\wedge(\alpha^{\flat})^{(k)}$ denote the curvature of the connection form $(\alpha^{\flat})^{(k)}$ on $\Delta^{k}\times\GG_{1}^{(k)}$ obtained as in Lemma \ref{con1}.  Let $P\in I^{*}_{q}(\RB)$.  The cochain $\psi_{\alpha^{\flat}}(P)$ in $\Omega^{*}(\GG_{1}^{(*)}/\SO(q,\RB))$ identifies in the same manner as in Example \ref{princ1} with the $\SO(q,\RB)$-basic cochain
	\begin{equation}\label{sumk}
	\sum_{k}\int_{\Delta^{k}}P((R^{\flat})^{(k)}),
	\end{equation}
	in $\Omega^{*}(\GG_{1}^{(*)})$.  Thus it suffices to show that the cochain in Equation \eqref{sumk} is zero.
	
	The form $(R^{\flat})^{(k)}$ is by construction of degree at most 1 in the $\Delta^{k}$ variables due to the $d(\alpha^{\flat})^{(k)}$, and therefore $P((R^{\flat})^{(k)})$ is of degree at most $\deg(P)$ in the $\Delta^{k}$ variables.  Thus $\int_{\Delta^{k}}P((R^{\flat})^{(k)})$ vanishes when $\deg(P)<k$, implying that $\psi_{\alpha}(P)$ vanishes in $\Omega^{*}(\GG^{(k)}_{1})$ for $k>\deg(P)$.
	
	Let us assume therefore that $k\leq \deg(P)$.  We will show that $\int_{\Delta^{k}}P((R^{\flat})^{(k)}) =0$ as a differential form on $\Omega^{2\deg(P)-k}(\GG^{(k)}_{1})$.  On $\Delta^{k}\times\GG^{(k)}_{1}$, using Equation \eqref{bob}, we compute
	\begin{equation}\label{curvk}
	(R^{\flat})^{(k)} = \sum_{i=0}^{k}dt_{i}\wedge(p^{k}_{i})^{*}\alpha^{\flat}+\sum_{i=0}^{k}t_{i}(p^{k}_{i})^{*}d\alpha^{\flat}+\bigg(\sum_{i=0}^{k}t_{i}(p^{k}_{i})^{*}\alpha^{\flat}\bigg)\wedge\bigg(\sum_{i=0}^{k}t_{i}(p^{k}_{i})^{*}\alpha^{\flat}\bigg),
	\end{equation}
	with the $p^{k}_{i}:\GG^{(k)}_{1}\rightarrow\Fr^{+}(N)$ defined as in Construction \ref{pki}.  To proceed further, we must consider a local coordinate picture.
	
	About a point $(u_{1},\dots,u_{k})\cdot \phi\in\GG^{(k)}_{1}$, consider a local coordinate chart for $\GG^{(k)}_{1}$ of the form
	\[
	\bigg((x^{j}_{1})_{j=1}^{\dim(\FF)};\dots;(x^{j}_{k})_{j=1}^{\dim(\FF)};(z^{j})_{j=1}^{q};g\bigg)\in B_{1}\times\cdots\times B_{k}\times V\times\GL^{+}(q,\RB),
	\]
	where the $B_{1},\dots,B_{k}$ are open balls in $\RB^{\dim(\FF)}$ corresponding to plaques in foliated charts $U_{1},\dots, U_{k}$ in $(M,\FF)$, and where $V$ is an open ball in $\RB^{q}$ such that $B_{k}\times V\cong U_{k}$.  For $\tu_{j+1} = u_{j+1}\cdots u_{k}\in\GG$ we let $(h_{\tu_{j+1}})^{l}:V\rightarrow\RB$ denote the $l^{th}$ component function of some holonomy transformation $h_{\tu_{j+1}}$ representing $\tu_{j+1}$.  Then in these coordinates the maps $p^{k}_{i}:\GG^{(k)}_{1}\rightarrow \Fr^{+}(N)$ take the form
	\[
	p^{k}_{i}\bigg((x^{j}_{1})_{j=1}^{\dim(\FF)};\dots;(x^{j}_{k})_{j=1}^{\dim(\FF)};(z^{j})_{j=1}^{q};g\bigg):= \bigg((x^{j}_{i+1})_{j=1}^{\dim(\FF)};\big((h_{\tu_{i+1}})^{j}(z^{1},\dots,z^{q})\big)_{j=1}^{q};g\bigg).
	\]
	To write $(R^{\flat})^{(k)}$ in these local coordinates, consider the chart $U_{i}\times\GL^{+}(q,\RB)$ of $\Fr^{+}(N)$.  In the foliated chart $U_{i}$ we have the local connection form $\alpha_{i}\in\Omega^{1}(U_{i};\mathfrak{gl}(q,\RB))$ corresponding to $\nabla^{\flat}$, which by Proposition \ref{bottconnections} vanishes on plaquewise\footnote{i.e. locally leafwise} tangent vectors.  Letting $\pi_{1}:U_{i}\times\GL^{+}(q,\RB)\rightarrow U_{i}$ and $\pi_{2}:U_{i}\times\GL^{+}(q,\RB)\rightarrow\GL^{+}(q,\RB)$ denote the projections, over $U_{i}\times\GL^{+}(q,\RB)$ the form $\alpha^{\flat}$ can be written
	\[
	\alpha^{\flat}_{(x,g)} = \Ad_{g^{-1}}\big(\pi_{1}^{*}\alpha_{i}\big)_{(x,g)}+\big(\pi_{2}^{*}\omega^{MC}\big)_{(x,g)},\hspace{7mm}(x,g)\in U_{i}\times\GL^{+}(q,\RB)
	\]
	where $\omega^{MC}$ is the Maurer-Cartan form on $\GL^{+}(q,\RB)$ \cite[Section 2.4 (b)]{gdf}.  For simplicity, let us abuse notation in letting $\alpha_{i}$ denote the form $\Ad^{-1}\big(\pi_{1}^{*}\alpha_{i}\big)$ on $U_{i}\times\GL^{+}(q,\RB)$.  Then by Proposition \ref{bottconnections} the matrix components of $\alpha_{i}$ can all be written in terms of the differentials of the transverse coordinates $z^{j}$ in $U_{i}$.  Consequently, in coordinates we can write
	\[
	(p^{k}_{i})^{*}\alpha^{\flat} = (p^{k}_{i})^{*}\alpha_{i}+\pi_{k}^{*}\omega^{MC},
	\]
	where $\pi_{k}:B_{1}\times\cdots\times B_{k}\times V\times\GL^{+}(q,\RB)\rightarrow\GL^{+}(q,\RB)$ is the projection and where $(p^{k}_{i})^{*}\alpha_{i}$ is a $\mathfrak{gl}(q,\RB)$-valued 1-form in the coordinate differentials $(dz^{j})_{j=1}^{q}$.
	
	Let us now rewrite the expression \eqref{curvk} for $(R^{\flat})^{(k)}$ in coordinates.  The first term on the right hand side can be written
	\begin{equation}\label{p1}
	\sum_{i=0}^{k}dt_{i}\wedge(p^{k}_{i})^{*}\alpha^{\flat} = \sum_{i=0}^{k}dt_{i}\wedge(p^{k}_{i})^{*}\alpha_{i}+\bigg(\sum_{i=0}^{k}dt_{i}\bigg)\wedge\pi_{k}^{*}\omega^{MC} = \sum_{i=0}^{k}dt_{i}\wedge(p^{k}_{i})^{*}\alpha_{i}
	\end{equation}
	since $\sum_{i=0}^{k}t_{i} = 1$.  The middle term on the right hand side of \eqref{curvk} can be written
	\begin{equation}\label{p2}
	\sum_{i=0}^{k}t_{i}(p^{k}_{i})^{*}d\alpha^{\flat} = \sum_{i=0}^{k}t_{i}(p^{k}_{i})\alpha_{i}+\bigg(\sum_{i=0}^{k}t_{i}\bigg)\pi_{k}^{*}d\omega^{MC} = \sum_{i=0}^{k}t_{i}(p^{k}_{i})\alpha_{i}+\pi_{k}^{*}d\omega^{MC},
	\end{equation}
	while the last term on the right hand side of \eqref{curvk} can be written
	\begin{align}
		\bigg(\sum_{i=0}^{k}t_{i}(p^{k}_{i})^{*}\alpha^{\flat}\bigg)\wedge\bigg(\sum_{i=0}^{k}t_{i}(p^{k}_{i})^{*}\alpha^{\flat}\bigg) =& \bigg(\sum_{i=0}^{k}t_{i}(p^{k}_{i})^{*}\alpha_{i}\bigg)\wedge\bigg(\sum_{i=0}^{k}t_{i}(p^{k}_{i})^{*}\alpha_{i}\bigg)\nonumber\\ &+ \bigg(\sum_{i=0}^{k}t_{i}(p^{k}_{i})^{*}\alpha_{i}\bigg)\wedge\pi_{k}^{*}\omega^{MC}\nonumber\\&+\pi_{k}^{*}\omega^{MC}\wedge\bigg(\sum_{i=0}^{k}t_{i}(p^{k}_{i})^{*}\alpha_{i}\bigg) + \pi_{k}^{*}(\omega^{MC}\wedge\omega^{MC})\label{p3}.
	\end{align}
	Adding the expressions \eqref{p1}, \eqref{p2} and \eqref{p3} and using the fact that the Maurer-Cartan form satisfies $d\omega^{MC}+\omega^{MC}\wedge\omega^{MC} = 0$ \cite[Equation 2.46]{gdf}, we find that
	\begin{align}
		(R^{\flat})^{(k)} =& \sum_{i=0}^{k}dt_{i}\wedge(p^{k}_{i})^{*}\alpha_{i}+\sum_{i=0}^{k}t_{i}(p^{k}_{i})^{*}d\alpha_{i}+\bigg(\sum_{i=0}^{k}t_{i}(p^{k}_{i})^{*}\alpha_{i}\bigg)\wedge\bigg(\sum_{i=0}^{k}t_{i}(p^{k}_{i})^{*}\alpha_{i}\bigg)\nonumber\\ &+\bigg(\sum_{i=0}^{k}t_{i}(p^{k}_{i})^{*}\alpha_{i}\bigg)\wedge\pi_{k}^{*}\omega^{MC}+\pi_{k}^{*}\omega^{MC}\wedge\bigg(\sum_{i=0}^{k}t_{i}(p^{k}_{i})^{*}\alpha_{i}\bigg)\label{p4}.
	\end{align}
	
	For a summand of $\int_{\Delta^{k}}P((R^{\flat})^{(k)})$ to be nonzero, it must contain precisely $k$ factors of the first term appearing in \eqref{p4}.  Thus, in our coordinates, due to the $(p^{k}_{i})^{*}\alpha_{i}$ appearing in this first term of \eqref{p4} each summand of $\int_{\Delta^{k}}P((R^{\flat})^{(k)})$ contains a string of wedge products of at least $k$ of the $dz^{i}$.  This consideration accounts for $2k$ of the coordinate differentials that appear in each summand of $\int_{\Delta^{k}}P((R^{\flat})^{(k)})$, and we must concern ourselves now with the $2\deg(P)-2k$ coordinate differentials that remain.
	
	Now each of the final four terms in \eqref{p4} is a matrix of 2-forms, and contains either an $\alpha_{i}$ or a $d\alpha^{i}$ as a factor.  Consequently, all the components of each such matrix must contain at least one $dz^{i}$ as a factor.  Therefore, of the remaining $2\deg(P)-2k$ coordinate differentials in each summand of $\int_{\Delta^{k}}P((R^{\flat})^{(k)})$, at least $\deg(P)-k$ more must be $dz^{i}$'s.  Thus in our local coordinate system for $\GG^{(k)}_{1}$, each summand in $\int_{\Delta^{k}}P((R^{\flat})^{(k)})$ contains a string of wedge products of at least $k+(\deg(P)-k) = \deg(P)>q$ of the $dz^{i}$, and must therefore be zero by dimension count.
\end{proof}

Bott's vanishing theorem at the level of the holonomy groupoid enables us to refine our characteristic map in a way entirely analogous to the classical case.

\begin{thm}\label{characteristicmap}
	If $\alpha^{\flat}\in\Omega^{1}(\Fr^{+}(N);\mathfrak{gl}(q,\RB))$ is a Bott connection on $N$, the cochain map $\psi^{\GG}_{\alpha^{\flat}}:WO_{q}\rightarrow\Omega^{*}(\GG_{1}^{(*)}/\SO(q,\RB))$ descends to a cochain map
	\[
	\phi^{\GG}_{\alpha^{\flat}}:\underline{WO}_{q}\rightarrow \Omega^{*}(\GG_{1}^{(*)}/\SO(q,\RB))
	\]
	whose induced map on cohomology is independent of the Bott connection chosen.\qed
\end{thm}

Recall now the characteristic map $\phi_{\alpha^{\flat}}:\underline{WO}_{q}\rightarrow \Omega^{*}(\Fr^{+}(N)/\SO(q,\RB))$ of Theorem \ref{classchar}, and for any $b\in\Omega^{*}(\GG^{(*)}_{1}/\SO(q,\RB))$ let $b_{0}$ denote its component in $\Omega^{*}(\Fr^{+}(N)/\SO(q,\RB))$.  Then by construction we have
\[
\big(\phi^{\GG}_{\alpha^{\flat}}(a)\big)_{0} = \phi_{\alpha^{\flat}}(a),\hspace{7mm}\text{for all $a\in\underline{WO}_{q}$}.
\]
Thus $\phi_{\alpha^{\flat}}$ should be thought of as encoding the ``static" transverse geometric information that can be accessed via classical Chern-Weil theory, while the ``larger" characteristic map $\phi^{\GG}_{\alpha^{\flat}}$ encodes both the static \emph{and dynamic} information pertaining to the relationship of the groupoid action with transverse geometry. As discussed in Remark \ref{pldn}, one can pull back the static information encoded by $\phi_{\alpha^{\flat}}$ to $\Omega^{*}(M)$ through the choice of a Euclidean structure for $N$, so one might hope that it is also possible to pull back all the dynamical information encoded by $\phi^{\GG}_{\alpha^{\flat}}$ to the double complex $\Omega^{*}(\GG^{(*)})$ in the same way.

Indeed it is claimed, with some vagueness, by Crainic and Moerdijk in \cite[Section 3.4]{crainic1} (who work with an \'{e}talified, \v{C}ech version of the double complex $\Omega^{*}(\GG^{(*)})$) that the contractibility of the fibres of $\Fr^{+}(N)/\SO(q,\RB)$ allows one to pull all of $\phi^{\GG}_{\alpha^{\flat}}(\underline{WO}_{q})$ down to $\Omega^{*}(\GG^{(*)})$ ``as in" the static case.  While it is unclear exactly what Crainic and Moerdijk mean by this, let us point out here that one is prevented from na\"{i}vely extending the cochain map $\sigma^{*}:\Omega^{*}(\Fr^{+}(N)/\SO(q,\RB))\rightarrow\Omega^{*}(M)$ to a cochain map $\Omega^{*}(\GG_{1}^{(*)}/\SO(q,\RB))\rightarrow\Omega^{*}(\GG^{(*)})$ precisely by the lack of invariance of the Euclidean structure on $N$ defining $\sigma$ under the action of $\GG$.  More precisely, we have the following proposition.

\begin{prop}
	Define $\sigma^{(k)}:\GG^{(k)}\rightarrow\GG^{(k)}_{1}/\SO(q,\RB)$ by the formula
	\[
	\sigma^{(k)}(u_{1},\dots,u_{k}):=(u_{1},\dots,u_{k})\cdot\sigma(s(u_{k})),\hspace{7mm}(u_{1},\dots,u_{k})\in\GG^{(k)}.
	\]
	Pulling back by the $\sigma^{(k)}$ defines a cochain map $\Omega^{*}(\GG_{1}/\SO(q,\RB))\rightarrow\Omega^{*}(\GG^{(*)})$ if and only if $\sigma$ is invariant under the action of $\GG$.
\end{prop}

\begin{proof}
	The $\sigma^{(k)}$ define a cochain map if and only if $\sigma^{(k-1)}\circ\epsilon^{k}_{i} = \epsilon^{k}_{i}\circ\sigma^{(k)}$ for all $i\leq k$.  However, we see that
	\[
	\epsilon^{k}_{k}\big(\sigma^{(k)}(u_{1},\dots,u_{k})\big) = (u_{1},\dots,u_{k})\cdot\big(u_{k}\cdot\sigma(s(u_{k}))\big)
	\]
	while
	\[
	\sigma^{(k-1)}\big(\epsilon^{k}_{k}(u_{1},\dots,u_{k})\big) = (u_{1},\dots,u_{k-1})\cdot\sigma(s(u_{k-1}))
	\]
	for all $(u_{1},\dots,u_{k})\in\GG^{(k)}$. Consequently we have $\epsilon^{k}_{k}\circ\sigma^{(k)} = \sigma^{(k-1)}\circ\epsilon^{k}_{k}$ if and only if $u\cdot\sigma(s(u)) = \sigma(r(u))$ for all $u\in\GG$, which occurs if and only if the Euclidean structure $\sigma$ on $N$ is preserved by the action of $\GG$.
\end{proof}

An invariant section $\sigma:M\rightarrow\Fr^{+}(N)/\SO(q,\RB)$ is the same thing as a $\GG$-invariant Euclidean structure on $N$, which is not always guaranteed to exist.  Moreover in any situation where such an invariant Euclidean structure \emph{does} exist, it induces via its determinant an invariant transverse volume form.  In this case, well-known results \cite[Theorem 2]{hurdkat} state that all generalised Godbillon-Vey classes (that is, those classes determined by cocycles $h_{1}h_{I}c_{J}\in\underline{WO}_{q}$, for multi-indices $I$ and $J$ with $\deg(c_{J}) = 2q$ and with $h_{I} =1$ permitted) vanish in $H^{*}_{dR}(M)$.  In particular, whenever $(M,\FF)$ has \emph{nonvanishing} Godbillon-Vey invariant, in order to probe the algebraic topology of $\GG$ using the characteristic map $\phi^{\GG}_{\alpha^{\flat}}$ we need a more sophisticated method of getting from $\Omega^{*}(\GG^{(*)}_{1}/\SO(q,\RB))$ to $\Omega^{*}(\GG^{(*)})$ which takes into account the lack of invariance of Euclidean structures on $N$ under the action of $\GG$.  Giving such a construction constitutes an interesting research question, which we leave to a future paper.

\section{The codimension 1 Godbillon-Vey cyclic cocycle}\label{sc3}

Connes and Moscovici \cite[Section 4]{diffcyc} use the \'{e}tale picture of a foliation groupoid to obtain an analogue of Theorem \ref{characteristicmap}.  More specifically, they replace $\GG_{1}$ with the groupoid $FX\rtimes\Gamma_{X}$ of germs of local diffeomorphisms of an $q$-manifold $X$, lifted to the frame bundle $FX$ of $X$.  Then they obtain a characteristic map from $H^{*}(\underline{WO}_{q})$ to the cyclic cohomology of the algebra $C_{c}^{\infty}(FX)\rtimes\Gamma_{X}$.  While unfortunately the lack of an easily-defined ``transverse exterior derivative" prevents a complete replication of the Connes-Moscovici construction in the non-\'{e}tale case, we can use Theorem \ref{characteristicmap} to give, in codimension 1, a cyclic cocycle for the Godbillon-Vey invariant on the algebra $C_{c}^{\infty}(\GG_{1};\Omega^{\frac{1}{2}})$ (recall from \cite{folops} that $C_{c}^{\infty}(\GG_{1};\Omega^{\frac{1}{2}})$ is the convolution algebra spanned by leafwise half-densities that are smooth with compact support in some Hausdorff open subset of the locally Hausdorff Lie groupoid $\GG_{1}$).

Let us begin with a preliminary calculation.  Suppose that $(M,\FF)$ is of codimension 1 (in which case $\SO(1,\RB)$ is the trivial group so we need not concern ourselves with basic elements), and let $\alpha\in\Omega^{1}(\Fr^{+}(N))$ correspond to a Bott connection on $N$ (we have dropped the $\flat$ superscript for notational simplicity).  We obtain the corresponding connection forms $\alpha^{(0)} = \alpha$ on $\GG^{(0)}_{1} = \Fr^{+}(N)$ and $\alpha^{(1)}$ on $\Delta^{1}\times\GG^{(1)}_{1}$ defined by
\[
\alpha^{(1)}_{(t;u)}:=t(p^{1}_{0})^{*}\alpha+(1-t)(p^{1}_{1})^{*}\alpha = t\,r^{*}\alpha+(1-t)\,s^{*}\alpha
\]
for $(t;u)\in\Delta^{1}\times\GG^{(1)}_{1}$.  For simplicity let us denote $(p^{1}_{i})^{*}\alpha$ by simply $\alpha_{i}$, $i=0,1$.  Then since $q=1$, the curvature of $\alpha^{(1)}$ is given simply by
\[
R^{(1)}_{(t;u)}:=dt\wedge(\alpha_{0}-\alpha_{1})+td\alpha_{0}+(1-t)d\alpha_{1}.
\]
Now in $\underline{WO}_{1}$ the Godbillon-Vey invariant is given by the cocycle $h_{1}c_{1}$, which is mapped via the $\phi^{\GG}_{\alpha}$ of Theorem \ref{characteristicmap} to the simplicial differential form
\begin{align*}
	\alpha^{(1)}&\wedge R^{(1)} = (t\alpha_{0}+(1-t)\alpha_{1})\wedge(dt\wedge(\alpha_{0}-\alpha_{1})+td\alpha_{0}+(1-t)d\alpha_{1})\\ =& -dt\wedge (t\alpha_{0}+(1-t)\alpha_{1})\wedge(\alpha_{0}-\alpha_{1})+(t\alpha_{0}+(1-t)\alpha_{1})\wedge(td\alpha_{0}+(1-t)d\alpha_{1})
\end{align*}
on $\Delta^{1}\times\GG^{(1)}_{1}$.  Integration over $\Delta^{1}$ then produces the form
\begin{align}
	\int_{0}^{1}\alpha^{(1)}\wedge R^{(1)} &= -\int_{0}^{1}tdt\wedge\alpha_{0}\wedge(\alpha_{0}-\alpha_{1})-\int_{0}^{1}(1-t)dt\wedge\alpha_{1}\wedge(\alpha_{0}-\alpha_{1})\nonumber\\ &= -\frac{1}{2}(\alpha_{0}+\alpha_{1})\wedge(\alpha_{0}-\alpha_{1})\label{intcurvmot}
\end{align}
on $\GG_{1}$.  Equation \eqref{intcurvmot} is geometrically opaque, and our immediate task now is to elucidate its geometric content.  First, we will prove that the factor $\alpha_{0}-\alpha_{1}$ has an interpretation as a path integral of the Bott curvature form $R$.

\begin{prop}\label{integralprop}
	Let $(M,\FF)$ be codimension $q$, and let $\alpha\in\Omega^{1}(\Fr^{+}(N);\mathfrak{gl}(q,\RB))$ correspond to a Bott connection on $N$ with associated curvature $R\in\Omega^{2}(\Fr^{+}(N);\mathfrak{gl}(q,\RB))$.  For $u\in\GG_{1}$, let $\gamma:[0,1]\rightarrow\Fr^{+}(N)$ be any smooth path in a leaf of $\FF_{\Fr^{+}(N)}$ that represents $u$. Letting $p:T\Fr^{+}(N)\rightarrow N_{\Fr^{+}(N)}$ denote the projection, for any $X\in T_{u}\GG_{1}$ choose a smooth vector field $\tilde{X}\in\Gamma^{\infty}(\gamma([0,1]);T\Fr^{+}(N))$ along $\gamma$ for which
	\begin{enumerate}
		\item $ds_{u}X =\tilde{X}_{\gamma(0)}$ and $dr_{u}X = \tilde{X}_{\gamma(1)}$, and 
		\item the projection $Z = p\tilde{X}\in\Gamma^{\infty}(\gamma([0,1]);N_{\Fr^{+}(N)})$ of $\tilde{X}$ to a normal vector field is parallel along $\gamma$ with respect to the Bott connection $\nabla^{\Fr^{+}(N)}$ (see Proposition \ref{bottconnections}) for the foliation $\FF_{\Fr^{+}(N)}$ determined by $\alpha$.
	\end{enumerate}
	Then 
	\begin{equation}\label{intcurvonfr}
	(\alpha_{0}-\alpha_{1})_{u}(X) = \int_{\gamma} R(\dot{\gamma},\tilde{X}).
	\end{equation}
	In particular, the integral on the right hand side does not depend on the choices of $\gamma$ and $\tilde{X}$.
\end{prop}

\begin{proof}
	That such a vector field $\tilde{X}$ can be chosen is a consequence of the surjectivity of the projection $p$ together with the definition of the parallel transport map for $N_{\Fr^{+}(N)}$ along $\gamma$.  More precisely, since $\GG$ acts on $N$ by parallel transport with respect to $\nabla^{\flat}$, the projections of $X_{u}$ to a normal vector on $N$ via the range and source are mapped to one another by parallel transport along any path representing $u$.  We compute
	\[
	R(\dot{\gamma},\tilde{X}) = d\alpha(\dot{\gamma},\tilde{X})+(\alpha\wedge\alpha)(\dot{\gamma},\tilde{X}) = \dot{\gamma}\alpha(\tilde{X})-\tilde{X}\alpha(\dot{\gamma})-\alpha([\dot{\gamma},\tilde{X}])+(\alpha\wedge\alpha)(\dot{\gamma},\tilde{X}).
	\]
	Since $\alpha$ is a Bott connection, all the leafwise tangent vectors $\dot{\gamma}$ lie in the kernel of $\alpha$ and so we can simplify to
	\[
	R(\dot{\gamma},\tilde{X}) = \dot{\gamma}\alpha(\tilde{X})-\alpha([\dot{\gamma},\tilde{X}]).
	\]
	By definition of a Bott connection we moreover have 
	\[
	[\dot{\gamma},\tilde{X}] = \nabla^{\Fr^{+}(N)}_{\dot{\gamma}}(Z)+[\dot{\gamma},\tilde{X}_{T\FF_{\Fr^{+}(N)}}] = [\dot{\gamma},\tilde{X}_{T\FF_{\Fr^{+}(N)}}]
	\]
	since $Z$ is parallel along $\gamma$, and where $\tilde{X}_{T\FF_{\Fr^{+}(N)}}$ is the leafwise component of $\tilde{X}$.  Since $T\FF_{\Fr^{+}(N)}$ is closed under brackets, $[\dot{\gamma},\tilde{X}]$ is also annihilated by $\alpha$ and we have
	\[
	R(\dot{\gamma},\tilde{X}) = \dot{\gamma}\alpha(\tilde{X}).
	\]
	Therefore by Stokes' theorem
	\[
	\int_{\gamma}R(\dot{\gamma},\tilde{X}) = \alpha_{\gamma(1)}(\tilde{X}_{\gamma(1)})-\alpha_{\gamma(0)}(\tilde{X}_{\gamma(0)}) = (\alpha_{0}-\alpha_{1})_{u}(X)
	\]
	as claimed.
\end{proof}

Next we show that the number $(\alpha_{0}-\alpha_{1})_{u}(X)$ defined for $X\in T_{u}\GG_{1}$ depends only on the projection of $X$ to a vector in the normal bundle $N$ of $(M,\FF)$, obtained via the range or source.

\begin{prop}\label{integralprop1}
	Let $\pi_{\GG_{1}}:\GG_{1}\rightarrow\GG$ be the projection induced by $\pi_{\Fr^{+}(N)}:\Fr^{+}(N)\rightarrow M$, and let $T_{r}\GG_{1}$ and $T_{s}\GG_{1}$ denote the tangent bundles to the range and source fibres of $\GG_{1}$ respectively, so that the differentials of $r\circ\pi_{\GG_{1}}$ and $s\circ\pi_{\GG_{1}}$ define fibrewise isomorphisms
	\[
	N_{1}:=T\GG_{1}/\big(T_{r}\GG_{1}\oplus T_{s}\GG_{1}\oplus\ker(d\pi_{\GG_{1}})\big)\rightarrow N.
	\]
	Then the formula in Equation \eqref{intcurvonfr} depends only on the class $[X]\in (N_{1})_{u}$ determined by $X$, and not on the choices of $\gamma$ and $\tilde{X}$.
\end{prop}

\begin{proof}
	To see that $(\alpha_{0}-\alpha_{1})_{u}(X)$ depends only on the class of $X$ in $N_{1}$ we consider a perturbation $X' = X+Y$ of $X$ where $Y\in\ker(d\pi_{\GG_{1}})_{u}$.  Identify $\ker(d\pi_{\GG_{1}})$ with $\GG\ltimes(\ker(d\pi_{\Fr^{+}(N)})) = \GG\ltimes(\Fr^{+}(N)\times\mathfrak{gl}(q,\RB))$.  By commutativity of the action of $\GG$ on $\Fr^{+}(N)$ with that of $\GL^{+}(q,\RB)$, the action of $\GG$ on the $\mathfrak{gl}(q,\RB)$ factor is by the identity, and we have $dr_{u}(Y) = ds_{u}(Y) = Y$.  Since $\alpha$ is a connection form we have $\alpha(Y) = Y$ and therefore
	\[
	(\alpha_{0}-\alpha_{1})_{u}(X+Y) = (\alpha_{0}-\alpha_{1})_{u}(X)+Y-Y = (\alpha_{0}-\alpha_{1})_{u}(X).
	\]
	Now suppose that $X' = X+Z$, where $Z\in T(\GG_{1}^{r(u)})\oplus T((\GG_{1})_{s(u)})$.  Then both $dr_{u}Z$ and $ds_{u}Z$ are contained in $T\FF_{\Fr^{+}(N)}$ and therefore are annihilated by $\alpha$.  Hence
	\[
	(\alpha_{0}-\alpha_{1})_{u}(X+Z) = (\alpha_{0}-\alpha_{1})_{u}(X)
	\]
	as required.
\end{proof}

In the paper \cite{chgpd}, the differential form
\[
\partial\alpha = (p^{1}_{0})^{*}\alpha-(p^{1}_{1})^{*}\alpha = (\epsilon^{1}_{0})^{*}\alpha-(\epsilon^{1}_{1})^{*}\alpha = r^{*}\alpha-s^{*}\alpha\in\Omega^{1}(\GG_{1};\mathfrak{gl}(q,\RB))
\]
determined by a connection form $\alpha$ appears as a measure of the failure of the connection form $\alpha$ to be invariant under the action of $\GG_{1}$.  In light of Proposition \ref{integralprop} we give any such differential form arising from a Bott connection a special name.

\begin{defn}\label{intcurvdef}
	Given a Bott connection form $\alpha\in\Omega^{1}(\Fr^{+}(N);\mathfrak{gl}(q,\RB))$, we refer to the 1-form
	\[
	R^{\GG}:=\partial\alpha = (r^{*}\alpha-s^{*}\alpha) \in\Omega^{1}(\GG_{1};\mathfrak{gl}(q,\RB))
	\]
	as the \textbf{integrated curvature} of $\alpha$.
\end{defn}

\begin{rmk}\normalfont
	Path integrals of differential forms such as in Proposition \ref{integralprop} are already of great use in determining de Rham representatives for loop space cohomology \cite{chen1, chen2, gjp, wilson}.  Since the holonomy groupoid is really a coarse sort of ``path space", in light of Proposition \ref{integralprop} one expects it to be possible to obtain a characteristic map for the holonomy groupoid defined in terms of iterated path integrals.  This would provide an exciting new geometric window into the existing theory of foliations, and has the potential to open up links with loop space theory.  We leave this question to a future paper.
\end{rmk}

Let us now come back to the Godbillon-Vey invariant of a codimension 1 foliated manifold $(M,\FF)$.  Denoting $\alpha^{\GG}:=-\frac{1}{2}(r^{*}\alpha+s^{*}\alpha)$ for notational simplicity, the differential form in \eqref{intcurvmot} can now be written
\[
\int_{0}^{1}\alpha^{(1)}\wedge R^{(1)}  = \alpha^{\GG}\wedge R^{\GG}\in\Omega^{2}(\GG_{1}^{(1)}).
\]
Thus we have reconciled the Chern-Weil description of the Godbillon-Vey invariant, as ``Bott connection wedge curvature", with the image of the Godbillon-Vey invariant arising from the characteristic map of Theorem \ref{characteristicmap}.  Proposition \ref{integralprop1} now allows us to integrate against $\alpha^{\GG}\wedge R^{\GG}$ in the following way.

\begin{lemma}\label{gvdensity}
	Let $(M,\FF)$ be a transversely orientable foliated $n$-manifold of codimension 1, and let $a^{0},a^{1}\in C_{c}^{\infty}(\GG_{1};\Omega^{\frac{1}{2}})$.  Then, setting $x = \pi_{\Fr^{+}(N)}(\phi)$ for $\phi\in\Fr^{+}(N)$, the formula
	\[
	gv(a^{0},a^{1})_{\phi}:=\int_{u\in\GG_{x}}a^{0}(u^{-1},u\cdot\phi)a^{1}(u,\phi)(\alpha^{\GG}\wedge R^{\GG})_{(u,\phi)}
	\]
	defines a compactly supported 1-density $gv(a^{0},a^{1})\in\Gamma(|\Fr^{+}(N)|)$.
\end{lemma}

\begin{proof}
	Fix $\phi\in\Fr^{+}(N)$.  For each $(u,\phi)\in\GG_{1}$, we have
	\[
	a^{0}(u^{-1},u\cdot\phi)a^{1}(u,\phi)\in|\FF_{\Fr^{+}(N)}|_{\phi}\otimes|\FF_{\Fr^{+}(N)}|_{u\cdot\phi},
	\]
	while $(\alpha^{\GG}\wedge R^{\GG})_{(u,\phi)}\in\Lambda^{2}(T^{*}_{(u,\phi)}\GG_{1})$, which we now describe using coordinates.
	
	Consider a chart $B_{1}\times B_{2}\times V\times\RB^{*}_{+}$ for $\GG_{1}$ about $(u,\phi)$, where $B_{1},B_{2}$ are open balls in $\RB^{\dim(\FF)}$ and $V$ an open ball in $\RB$ such that $B_{2}\times V\cong U_{2}\subset M$ is a foliated chart about $\pi_{\Fr^{+}(N)}(\phi)$ and $B_{1}\times h_{u}(V)\cong U_{1}$ is a foliated chart about $\pi_{\Fr^{+}(N)}(u\cdot\phi)$, where $h_{u}:V\rightarrow h_{u}(V)\subset\RB$ is a holonomy diffeomorphism representing $u$.  By Proposition \ref{integralprop1}, in the local coordinates $((x^{i}_{1})_{i=1}^{\dim{\FF}};(x^{i}_{2})_{i=1}^{\dim(\FF)};z;t)\in B_{1}\times B_{2}\times V\times\RB^{*}_{+}$ we have that $R^{\GG} = f_{1}dz$ for some $f_{1}$ defined on $B_{1}\times B_{2}\times V\times\RB^{*}_{+}$.  Moreover, with $t^{-1}dt$ the Maurer-Cartan form on $\RB^{*}_{+}$, $\alpha^{\GG}$ is of the form $f_{2}dz+t^{-1}dt$ for some smooth $f_{2}$ defined on $B_{1}\times B_{2}\times V$.  Consequently, $\alpha^{\GG}\wedge R^{\GG}$ is of the form $ft^{-1}dt\wedge dz$ for some smooth function $f$ defined on $B_{1}\times B_{2}\times V\times\RB^{*}_{+}$.
	
	Since the coordinate differentials $dz$ and $dt$ span $T_{\phi}^{*}\Fr^{+}(N)\ominus T_{\phi}^{*}\FF_{\Fr^{+}(N)}$, at each point $(u,\phi)\in\GG_{1}$ we have that
	\[
	a^{0}(u^{-1},u\cdot\phi)a^{1}(u,\phi)(\alpha^{\GG}\wedge R^{\GG})_{(u,\phi)}\in|\Fr^{+}(N)|_{\phi}\otimes|T\FF_{\Fr^{+}(N)}|_{u\cdot\phi}.
	\]
	Then by compact support of $a^{0}$ and $a^{1}$, the integral
	\[
	gv(a^{0},a^{1})_{\phi} = \int_{u\in\GG_{x}}a^{0}(u^{-1},u\cdot\phi)a^{1}(u,\phi)(\alpha^{\GG}\wedge R^{\GG})_{(u,\phi)}\in|\Fr^{+}(N)|_{\phi}
	\]
	is well-defined and $gv(a^{0},a^{1})$ is a compactly supported density on $\Fr^{+}(N)$.
\end{proof}

\begin{thm}\label{gvcocycle}
	Let $(M,\FF)$ be a transversely orientable foliated $n$-manifold of codimension 1.  Then for $a^{0},a^{1}\in C_{c}^{\infty}(\GG_{1};\Omega^{\frac{1}{2}})$ the formula
	\[
	\varphi_{gv}(a^{0},a^{1}):=\int_{\phi\in\Fr^{+}(N)}gv(a^{0},a^{1})_{\phi} = \int_{(u,\phi)\in\GG_{1}}a^{0}(u^{-1},u\cdot\phi)a^{1}(u,\phi)(\alpha^{\GG}\wedge R^{\GG})_{(u,\phi)}
	\]
	defines a cyclic 1-cocycle $\varphi_{gv}$ on the convolution algebra $C_{c}^{\infty}(\GG_{1};\Omega^{\frac{1}{2}})$.
\end{thm}

\begin{proof}
	We work with Connes' $\lambda$-complex \cite{ncg}.  For notational simplicity we denote elements of $\GG_{1}$ by $v_{i}$, and we use the notation $\int_{v_{0}v_{1}v_{2}\in\Fr^{+}(N)}$ to mean the iterated integral over all triples $(v_{0},v_{1},v_{2})\in\GG_{1}^{(3)}$ for which $v_{0}v_{1}v_{2}\in\Fr^{+}(N)$, followed by an integral over $\Fr^{+}(N)$.  For $a^{0},a^{1},a^{2}\in C_{c}^{\infty}(\GG_{1};\Omega^{\frac{1}{2}})$, we calculate
	\begin{align*}
		\varphi_{gv}(a^{0}a^{1},a^{2}) =& \int_{v_{0}v_{1}v_{2}\in\Fr^{+}(N)}a^{0}(v_{0})a^{1}(v_{1})a^{2}(v_{2})(\alpha^{\GG}\wedge R^{\GG})_{v_{2}}\\ =&\int_{v_{0}v_{1}v_{2}\in\Fr^{+}(N)}a^{0}(v_{0})a^{1}(v_{1})a^{2}(v_{2})\,(\epsilon^{2}_{0})^{*}(\alpha^{\GG}\wedge R^{\GG})_{(v_{1},v_{2})},
	\end{align*}
	\begin{align*}
		\varphi_{gv}(a^{0},a^{1}a^{2}) =& \int_{v_{0}v_{1}v_{2}\in\Fr^{+}(N)}a^{0}(v_{0})a^{1}(v_{1})a^{2}(v_{2})(\alpha^{\GG}\wedge R^{\GG})_{v_{1}v_{2}}\\ =&\int_{v_{0}v_{1}v_{2}\in\Fr^{+}(N)}a^{0}(v_{0})a^{1}(v_{1})a^{2}(v_{2})\,(\epsilon^{2}_{1})^{*}(\alpha^{\GG}\wedge R^{\GG})_{(v_{1},v_{2})}
	\end{align*}
	and
	\begin{align*}
		\varphi_{gv}(a^{2}a^{0},a^{1}) =& \int_{v_{2}v_{0}v_{1}\in\Fr^{+}(N)}a^{2}(v_{2})a^{0}(v_{0})a^{1}(v_{1})(\alpha^{\GG}\wedge R^{\GG})_{v_{1}}\\ =& \int_{v_{0}v_{1}v_{2}\in\Fr^{+}(N)}a^{0}(v_{0})a^{1}(v_{1})a^{2}(v_{2})\,(\epsilon^{2}_{2})^{*}(\alpha^{\GG}\wedge R^{\GG})_{(v_{1},v_{2})}.
	\end{align*}
	Because $h_{1}c_{1}\in\underline{WO}_{1}$ is closed under $d$, the component $\alpha^{\GG}\wedge R^{\GG}\in\Omega^{2}(\GG_{1}^{(1)})$ of its image under the cochain map $\psi_{\alpha}:\underline{WO}_{1}\rightarrow\Omega^{*}(\GG^{(*)}_{1})$ of Theorem \ref{characteristicmap} is closed under $\partial:\Omega^{2}(\GG^{(1)}_{1})\rightarrow\Omega^{2}(\GG^{(2)}_{1})$.  Thus
	\begin{align*}
		b\varphi_{gv}(a^{0},a^{1},a^{2}) =& \varphi_{gv}(a^{0}a^{1},a^{2})-\varphi_{gv}(a^{0},a^{1}a^{2})+\varphi_{gv}(a^{2}a^{0},a^{1})\\ =&  \int_{v_{0}v_{1}v_{2}\in\Fr^{+}(N)}a^{0}(v_{0})a^{1}(v_{1})a^{2}(v_{2})\,\partial(\alpha^{\GG}\wedge R^{\GG})_{(v_{1},v_{2})}\\ =& 0
	\end{align*}
	making $\varphi_{gv}$ a Hochschild cocycle.
	
	It remains only to check that $\varphi_{gv}(a^{0},a^{1}) = -\varphi_{gv}(a^{1},a^{0})$.  For this, we observe that by definition $R^{\GG}_{\phi} = \alpha_{\phi}-\alpha_{\phi} = 0$ for any unit $\phi\in\GG_{1}$, hence
	\[
	0 = \partial(\alpha^{\GG}\wedge R^{\GG})_{(v^{-1},v)} = (\alpha^{\GG}\wedge R^{\GG})_{v}-(\alpha^{\GG}\wedge R^{\GG})_{v^{-1}v}+(\alpha^{\GG}\wedge R^{\GG})_{v^{-1}} = (\alpha^{\GG}\wedge R^{\GG})_{v^{-1}}+(\alpha^{\GG}\wedge R^{\GG})_{v}.
	\]
	Therefore
	\begin{align*}
		\varphi_{gv}(a^{0},a^{1}) =& \int_{v\in\GG_{1}}a^{0}(v^{-1})a^{1}(v)(\alpha^{\GG}\wedge R^{\GG})_{v} = -\int_{v^{-1}\in\GG_{1}}a^{1}(v)a^{0}(v^{-1})(\alpha^{\GG}\wedge R^{\GG})_{v^{-1}}\\ =& -\varphi_{gv}(a^{1},a^{0})
	\end{align*}
	making $\varphi_{gv}$ a cyclic cocycle.
\end{proof}

\begin{defn}
	We refer to the cyclic cocycle $\varphi_{gv}$ on $C_{c}^{\infty}(\GG_{1};\Omega^{\frac{1}{2}})$ given in Theorem \ref{gvcocycle} as the \textbf{Godbillon-Vey cyclic cocycle}.
\end{defn}

\begin{rmk}\normalfont\label{whatsthepoint}The Godbillon-Vey cyclic cocycle for $C_{c}^{\infty}(\GG_{1};\Omega^{\frac{1}{2}})$ is the analogue of the Connes-Moscovici formula \cite[Proposition 19]{backindgeom} for the crossed product of a manifold by a discrete group action.  Note that in contrast with the \'{e}tale setting of Connes and Moscovici, the differential form $\alpha^{\GG}\wedge R^{\GG}$ on $\GG_{1}$ with respect to which $\varphi_{gv}$ is defined has, by Proposition \ref{integralprop}, an explicit interpretation in terms of the integral of the Bott curvature along paths representing elements in $\GG_{1}$.  Such a geometric interpretation is novel, and is completely lost in the \'{e}tale setting that has been almost exclusively used in studying the secondary characteristic classes of foliations using noncommutative geometry.
\end{rmk}

One has the following immediate corollary of Proposition \ref{integralprop} which, while  completely unsurprising, is novel due to our non-\'{e}tale perspective that incorporates the global transverse geometry of $(M,\FF)$.

\begin{cor}
	If $(M,\FF)$ is a codimension 1, transversely orientable foliated manifold with a flat Bott connection, then the Godbillon-Vey cyclic cocycle vanishes.\qed
\end{cor}

Our final task is to demonstrate that the Godbillon-Vey cyclic cocycle coincides with the cocycle obtained from the local index formula for the semifinite spectral triple considered in \cite[Section 4.3]{macr1}.  For this purpose, it will be convenient to have a formula for the Godbillon-Vey cyclic cocycle in terms of a transverse volume form $\omega\in\Omega^{1}(M)$ (that is, a form $\omega$ that is nonvanishing and is identically zero on leafwise tangent vectors).  By the final statement of Proposition \ref{integralprop1}, we know that we can write
\begin{equation}\label{Rgdelta}
R^{\GG} = \delta\, (s\circ \pi^{(1)})^{*}\omega
\end{equation}
for some smooth function $\delta:\GG_{1}\rightarrow\RB$, where $s:\GG\rightarrow M$ is the source and where $\pi^{(1)}:\GG_{1}\rightarrow\GG$ is the projection.  Since $r^{*}\omega$ and $s^{*}\omega$ both annihilate the tangents to the range and source fibres we can formulate the following definition.

\begin{defn}\label{modularfunction}
	Given a transverse volume form $\omega\in\Omega^{1}(M)$, the smooth homomorphism $\Delta:\GG\rightarrow\RB^{*}_{+}$ defined by the equation
	\[
	r^{*}\omega = \Delta\,s^{*}\omega
	\]
	is called the \textbf{modular function} or \textbf{Radon-Nikodym derivative} associated to $\omega$.
\end{defn}

Note that the face maps $\epsilon^{2}_{i}:\GG^{(2)}\rightarrow\GG$ of $\GG$ satisfy
\[
s\circ\epsilon^{2}_{0}(u_{1},u_{2}) = s\circ\epsilon^{2}_{1}(u_{1},u_{2}) = s(u_{2}),\hspace{5mm} s\circ\epsilon^{2}_{2}(u_{1},u_{2}) = r\circ\epsilon^{2}_{0}(u_{1},u_{2}) = r(u_{2})
\]
for all $(u_{1},u_{2})\in\GG^{(2)}$.  Via a mild abuse of notation let us also denote the face maps of $\GG_{1}$ by $\epsilon^{j}_{i}$.  Then the fact that $R^{\GG} = \partial\alpha$ gives $(\epsilon^{2}_{0})^{*}R^{\GG}-(\epsilon^{2}_{1})^{*}R^{\GG}+(\epsilon^{2}_{2})^{*}R^{\GG} = \partial^{2}\alpha = 0$.  Therefore, letting $\pi^{(2)}:\GG_{1}^{(2)}\rightarrow\GG^{(2)}$ denote the projection, we have
\begin{align*}
	0 =& \,\delta(u_{2})(s\circ \pi^{(1)}\circ\epsilon^{2}_{0})^{*}\omega_{(u_{1},u_{2})}-\delta(u_{1}u_{2})(s\circ\pi^{(1)}\circ\epsilon^{2}_{1})^{*}\omega_{(u_{1},u_{2})}\\& + \delta(u_{1})(s\circ\pi^{(1)}\circ\epsilon^{2}_{2})^{*}\omega_{(u_{1},u_{2})}\\ =&\, \delta(u_{2})(s\circ\epsilon^{2}_{0}\circ\pi^{(2)})^{*}\omega_{(u_{1},u_{2})}-\delta(u_{1}u_{2})(s\circ\epsilon^{2}_{1}\circ\pi^{(2)})^{*}\omega_{(u_{1},u_{2})}\\&+\delta(u_{1})(s\circ\epsilon^{2}_{2}\circ\pi^{(2)})^{*}\omega_{(u_{1},u_{2})}\\ =&\, \big(\delta(u_{1})-\delta(u_{1}u_{2})\big)(s\circ\epsilon^{2}_{0}\circ\pi^{(2)})^{*}\omega_{(u_{1},u_{2})}+\delta(u_{1})(r\circ\epsilon^{2}_{0}\circ\pi^{(2)})^{*}\omega_{(u_{1},u_{2})}\\ =&\, \big(\delta(u_{2})-\delta(u_{1}u_{2})+\delta(u_{1})\Delta(u_{2})\big)(s\circ\epsilon^{2}_{0}\circ\pi^{(2)})^{*}\omega_{(u_{1},u_{2})}
\end{align*}
for all $(u_{1},u_{2})\in\GG^{(2)}$.  Hence
\begin{equation}\label{bigdlild}
\delta(u_{1}u_{2}) = \delta(u_{2})+\delta(u_{1})\Delta(u_{2}),\hspace{5mm}(u_{1},u_{2})\in\GG^{(2)}.
\end{equation}
Now the choice of $\omega$ determines a trivialisation $\Fr^{+}(N)\cong M\times\RB^{*}_{+}$ in which we can write elements of $\GG_{1}$ as $(u,x,t)\in \GG\ltimes(M\times\RB^{*}_{+})$.  By the arguments of the second paragraph in the proof of Lemma \ref{gvdensity} we can now write
\[
(\alpha^{\GG}\wedge R^{\GG})_{(u,x,t)} = \frac{\delta(u)}{t}dt\wedge\omega_{x},\hspace{5mm}(u,x,t)\in\GG\ltimes(M\times\RB^{*}_{+})
\]
so that our Godbillon-Vey cyclic cocycle becomes
\[
\varphi_{gv}(a^{0},a^{1}) = -\int_{(x,t)\in M\times\RB^{*}_{+}}\int_{u\in\GG_{x}}a^{0}(u^{-1},\Delta(u)t)\,a^{1}(u,t)\,\frac{\delta(u)}{t}\,\omega_{x}\wedge dt
\]
for $a^{0},a^{1}\in C_{c}^{\infty}(\GG_{1};\Omega^{\frac{1}{2}})$.  In order to compare our formula with that in \cite{macr1}, we will want to assume that $\GG_{1}$ is of the form $\GG_{1} = \Fr^{+}(N)\rtimes\GG$ rather than $\GG\ltimes\Fr^{+}(N)$.  Now these groupoids are of course isomorphic via the map $\Fr^{+}(N)\rtimes\GG\ni(\phi,u)\mapsto(u,u^{-1}\cdot\phi)\in\GG\ltimes\Fr^{+}(N)$, and a function $a\in C_{c}^{\infty}(\GG\ltimes\Fr^{+}(N);\Omega^{\frac{1}{2}})$ identifies under this map with $\tilde{a}\in C_{c}^{\infty}(\Fr^{+}(N)\rtimes\GG;\Omega^{\frac{1}{2}})$ given by
\[
\tilde{a}(\phi,u):=a(u,u^{-1}\cdot\phi),\hspace{5mm}(\phi,u)\in\Fr^{+}(N)\rtimes\GG.
\]
With these identifications, and using the notational convention
\[
a_{u}(\phi):=a(\phi,u)\hspace{7mm}(\phi,u)\in\Fr^{+}(N)\rtimes\GG
\]
for $a\in C_{c}^{\infty}(\Fr^{+}(N)\rtimes\GG;\Omega^{\frac{1}{2}})$, we see that the Godbillon-Vey cyclic cocycle is defined for $a^{0},a^{1}\in C_{c}^{\infty}(\Fr^{+}(N)\rtimes\GG;\Omega^{\frac{1}{2}})$ by the formula
\begin{equation}\label{omegagvcocycle}
\varphi_{gv}(a^{0},a^{1}) = -\int_{(x,t)\in M\times\RB^{*}_{+}}\int_{u\in\GG^{x}}a^{0}_{u}(x,t)a^{1}_{u^{-1}}(u^{-1}\cdot x, \Delta(u^{-1})t)\frac{\delta(u^{-1})}{t}\omega_{x}\wedge dt.
\end{equation}
Let us now recall the cocycle $\phi_{1}$ obtained via the local index formula in \cite[Section 4.3]{macr1}.  In the coordinates we have chosen in this paper, $\phi_{1}$ is given by the equation
\begin{equation}\label{macrcocycle}
\phi_{1}(a^{0},a^{1}) = -(2\pi i)^{\frac{1}{2}}\int_{(x,t)\in M\times\RB^{*}_{+}}\int_{u\in\GG^{x}}a^{0}_{u}(x,t)\,a^{1}_{u^{-1}}(u^{-1}\cdot x,\Delta(u^{-1})t)\,\frac{\partial\log\Delta(u^{-1})}{t}\omega_{x}\wedge dt
\end{equation}
for $a^{0},a^{1}\in C_{c}^{\infty}(\GG_{1};\Omega^{\frac{1}{2}})$.  Thus in order to conclude that the index formula $\phi_{1}$ of \cite{macr1} and the cyclic cocycle $\varphi_{gv}$ of Equation \eqref{omegagvcocycle} coincide (up to the constant multiple $(2\pi i)^{\frac{1}{2}}$), we need only show that $\delta=\partial\log\Delta$.  This will be a consequence of the following fact, which holds for foliations of arbitrary codimension and gives a geometric interpretation (in the non-\'{e}tale setting) for the off-diagonal term appearing in the triangular structures considered by Connes \cite[Lemma 5.2]{cyctrans} and Connes-Moscovici \cite[Part I]{CM}.

\begin{prop}
	Let $(M,\FF)$ be a transversely orientable foliated manifold of codimension $q$, and let $\alpha^{\flat}\in\Omega^{1}(\Fr^{+}(N);\mathfrak{gl}(q,\RB))$ be a Bott connection form.  Let $H:=\ker(\alpha^{\flat})/T\FF_{\Fr^{+}(N)}$ be the horizontal normal bundle determined by $\alpha^{\flat}$, let $V:=\ker(d\pi_{\Fr^{+}(N)})$ be the vertical tangent bundle, and let
	\[
	N_{\Fr^{+}(N)} = V\oplus H \cong \Fr^{+}(N)\times (\mathfrak{gl}(q,\RB)\oplus\RB^{q})
	\]
	be the corresponding decomposition of $N_{\Fr^{+}(N)}$, with $V$ and $H$ trivialised as in the proof of Proposition \ref{bottconnections}.  With respect to this decomposition, for $u\in\GG$ and for any $\phi\in\Fr^{+}(N)_{s(u)}$ the action $u^{\Fr^{+}(N)}_{*}:(N_{\Fr^{+}(N)})_{\phi}\rightarrow (N_{\Fr^{+}(N)})_{u\cdot\phi}$ can be written
	\begin{equation}\label{matrixu}
	u^{\Fr^{+}(N)}_{*} = \begin{pmatrix} \id_{\mathfrak{gl}(q,\RB)} & R^{\GG}_{u} \\ 0 & \id_{\RB^{q}} \end{pmatrix}.
	\end{equation}
\end{prop}

\begin{proof}
	That the top left corner is $\id_{\mathfrak{gl}(q,\RB)}$ follows from the commutativity of the left action of $\GG$ on $\Fr^{+}(N)$ with the right action of $\GL^{+}(q,\RB)$, and the bottom left entry is zero for the same reason.  For the bottom right corner, we note that equivariance of the map $\pi_{\Fr^{+}(N)}:\Fr^{+}(N)\rightarrow M$ with respect to the action of $\GG$ implies that the induced fibrewise isomorphisms $\pi_{\phi}:H_{\phi}\rightarrow N_{\pi_{\Fr^{+}(N)}(\phi)}$ are also equivariant.  Thus if $[v_{\phi}]\in H_{\phi}$, denoting $u_{22}\cdot [v_{\phi}] :=\proj_{H}\big(u_{*}[v_{\phi}]\big)$ and letting $u_{*}:N_{s(u)}\rightarrow N_{r(u)}$ denote the action of $u$ on $N$, we have
	\[
	\big((u\cdot\phi)^{-1}\circ \pi_{u\cdot\phi}\big)(u_{22}\cdot [v_{\phi}]) = \big(\phi^{-1}\circ u_{*}^{-1}\circ u_{*}\circ\pi_{\phi}\big)([v_{\phi}]) = \big(\phi^{-1}\circ \pi_{\phi}\big)(v_{\phi})
	\]
	giving the bottom right entry of \eqref{matrixu}.
	
	Finally we come to the top right entry.  Since a connection form maps vertical vectors to themselves, the top right entry of \eqref{matrixu} is the map which sends $[v_{\phi}]\in H$ to $\alpha^{\flat}(u^{\Fr^{+}(N)}_{*}v_{\phi})$, where $u^{\Fr^{+}(N)}_{*}v_{\phi}$ is any element of $T\Fr^{+}(N)$ representing $u^{\Fr^{+}(N)}_{*}[v_{\phi}]$.  Since $v_{\phi}$ is contained in $\ker(\alpha^{\flat})$, we can equally regard the top right entry as the map which sends $[v_{\phi}]\in(\ker(\alpha^{\flat})/T\FF_{\Fr^{+}(N)})$ to
	\[
	\alpha^{\flat}(u^{\Fr^{+}(N)}_{*}v_{\phi})-\alpha^{\flat}(v_{\phi}),
	\]
	which by \eqref{intcurvonfr} in Proposition \ref{integralprop} coincides with $R^{\GG}_{u}(v_{\phi})$, the well-definedness of which is due to Proposition \ref{integralprop1}.
\end{proof}

Coming back to our codimension 1 foliation $(M,\FF)$, recall \cite[p. 23]{macr1} that the function $\partial\log\Delta$ on $\GG$ is by definition the top right corner of the matrix in Equation \eqref{matrixu} that gives the action of $\GG$ on $N_{\Fr^{+}(N)}$.  Therefore
\[
\delta = \partial\log\Delta
\]
as required, proving the following result.

\begin{thm}
	The Godbillon-Vey cyclic cocycle of Equation \eqref{omegagvcocycle} coincides with the local index formula cocycle in Equation \eqref{macrcocycle} for the semifinite spectral triple considered in \cite[Section 4.3]{macr1}.\qed
\end{thm}

	\bibliographystyle{amsplain}
	\bibliography{references}
	
\end{document}